\documentclass[11pt,reqno]{amsart}
\usepackage[margin=0.73in]{geometry}
\usepackage{amsmath,amssymb,amsthm,graphicx,amsxtra, setspace}
\usepackage[utf8]{inputenc}
\usepackage{mathrsfs}
\usepackage{hyperref}
\usepackage{upgreek}
\usepackage{mathtools}
\usepackage[mathcal]{euscript}
\usepackage{xcolor}
\usepackage{upref}
\usepackage{multirow}
\allowdisplaybreaks

\usepackage[pagewise]{lineno}

\newtheorem{theorem}{Theorem}[section]
\newtheorem{lemma}[theorem]{Lemma}
\newtheorem{proposition}[theorem]{Proposition}

\newtheorem{corollary}[theorem]{Corollary}
\newtheorem{definition}[theorem]{Definition}
\newtheorem{example}[theorem]{Example}
\newtheorem{remark}[theorem]{Remark}

\newtheorem{hypothesis}[theorem]{Hypothesis}

\let\originalleft\left
\let\originalright\right
\renewcommand{\left}{\mathopen{}\mathclose\bgroup\originalleft}
\renewcommand{\right}{\aftergroup\egroup\originalright}

\def\N{\mathbb{N}}

\def\u{\boldsymbol{\zeta}}
\def\v{\boldsymbol{\eta}}
\def\w{\boldsymbol{\mathfrak{z}}}
\def\f{\boldsymbol{f}}

\def\h{\theta}

\def\Drm{\mathrm{D}}
\def\Vbb{\mathbb{V}}
\def\Lbb{\mathbb{L}}
\def\Hbb{\mathbb{H}}

\def\Ocal{\mathcal{Q}}
\def\Pcal{\mathcal{P}}
\def\Acal{\mathcal{A}}
\def\Hcal{\mathcal{H}}

\def\Mrm{\mathrm{M}}
\def\Frm{\mathrm{F}}
\def\Rrm{\boldsymbol{\mathfrak{I}}}
\def\drm{\mathrm{d}}
\def\Urm{\mathcal{Z}}
\def\Wrm{\mathcal{W}}
\def\Grm{\boldsymbol{\Upsilon}}
\def\Lrm{\mathrm{L}}
\def\Qrm{\mathrm{Q}}

\def\Pbb{\mathbb{P}}
\def\Ebb{\mathbb{E}}
\def\Arm{\mathrm{A}}

\def\B{\mathcal{B}}

\def\K{\mathcal{K}}
\def\d{\mathrm{d}}
\def\C{\mathrm{C}}
\def\wi{\widetilde}
\def\Tr{\mathrm{Tr}}

\def\W{\mathrm{W}}

\numberwithin{equation}{section}

\newcommand{\R}{\mathbb{R}}

\renewcommand{\d}{\/\mathrm{d}\/}

\newcommand{\Addresses}{{
		\footnote{
			\noindent \textsuperscript{1}Center for Mathematics and Applications (NOVA Math), NOVA School of Science and Technology (NOVA FCT),	Portugal.\par\nopagebreak
			\noindent 
			
			\textit{e-mail:} \texttt{Kush Kinra: kushkinra@gmail.com, k.kinra@fct.unl.pt.}
			
			\noindent \textsuperscript{*}Corresponding author.
			
			\textit{Key words:} Stochastic continuous data assimilation, stochastic convective Brinkman-Forchheimer equations, additive and multiplicative noise, Foias-Prodi estimates in expected value.
			
			Mathematics Subject Classification (2020): Primary 60H15, 35R60, 76B75; Secondary 35Q30, 60H30, 37C50. 
			
}}}

\begin{document}
	
	\title[Continuous data assimilation for stochastic convective Brinkman-Forchheimer equations]{A note on continuous data assimilation for stochastic convective Brinkman-Forchheimer equations in 2D and 3D 
		\Addresses}
	
	\author[Kush Kinra]
	{Kush Kinra\textsuperscript{1*}}

	\begin{abstract}
Continuous data assimilation (CDA) methods, such as the nudging algorithm introduced by Azouani, Olson, and Titi (AOT) [\emph{Azouani et. al., J. Nonlinear Sci., \textbf{24} (2014), 277-304.}], have proven to be highly effective in deterministic settings for asymptotically synchronizing approximate solutions with observed dynamics. In this note, we introduce and analyze an algorithm for CDA for the two- and three-dimensional stochastic convective Brinkman-Forchheimer equations (CBFEs) driven by either additive or multiplicative Gaussian noise. The model is believed to provide an accurate description when the flow velocity exceeds the regime of validity of Darcy’s law and the porosity remains moderately large. We derive sufficient conditions on the nudging parameter and the spatial resolution of observations that ensure convergence of the assimilated solution to the true stochastic flow. We demonstrate convergence in the mean-square sense, and additionally establish pathwise convergence in the presence of additive noise. The CBFEs, also known as Navier-Stokes equations with damping, exhibit enhanced stability properties due to the presence of nonlinear damping term. In particular, we show that nonlinear damping not only enables the implementation of CDA in three dimensions but also yields improved convergence results in two dimensions when compared to the classical Navier-Stokes equations.
	\end{abstract}
	
	\maketitle


\section{Introduction}

\subsection{Mathematical Model and Solvability}
Fluid flow through porous media has attracted sustained attention due to its importance in a wide range of engineering and applied science applications. A porous medium or a porous material is defined as a solid (often called a matrix) permeated by an interconnected network of pores (voids) filled with a fluid (e.g., air or water). In this study, we consider the case of a porous medium that is completely saturated with a fluid. In particular, we consider stochastic convective Brinkman-Forchheimer equations (SCBFEs) which describe the motion of an incompressible fluid through a saturated porous medium. These equations yield a more accurate representation of the flow dynamics in regimes where the fluid velocity is sufficiently large that Darcy’s law alone is no longer applicable, and where the porosity is not negligibly small. For this reason, such models are commonly referred to as \emph{non-Darcy} models in the literature (see, for instance, \cite{GKM_AMOP,PAM}).

Let $\Ocal \subset \R^d$, with $d=2,3$, be a bounded domain with smooth boundary $\partial\Ocal$. We denote by $\u(x,t) \in \R^d$ the fluid velocity at spatial position $x$ and time $t$, by $p(x,t) \in \R$ the associated pressure field, and by $\f(x) \in \R^d$ an external body force. The stochastic incompressible system considered in this note is given by
\begin{equation}\label{1}
	\left\{
	\begin{aligned}
		\drm \u -\mu \Delta\u \drm t + (\u\cdot\nabla)\u \drm t + \alpha \u \drm t + \beta|\u|^{\varpi-1}\u \drm t +\nabla p \drm t &=\f \drm t + \Grm(t,\u)\drm\Wrm, && \text{ in } \ \Ocal\times(0,\infty), \\ \nabla\cdot\u&=0, \ &&\text{ in }  \Ocal\times[0,\infty), \\
		\u&=\mathbf{0},&& \text{ on } \ \partial\Ocal\times[0,\infty), \\
		\u(0)&=\u_0, && \text{ in } \ \Ocal,
	\end{aligned}
	\right.
\end{equation} 
where $\Wrm$ is a $\Qrm$-Wiener process  and $\Grm(\cdot,\cdot)$ is the noise coefficient (see Subsection \ref{Sto-set} for the details). To guarantee uniqueness of the pressure $p$, we additionally impose the normalization condition $\int_{\Ocal} p(x,t)\,\drm x = 0$, $t \in [0,\infty)$. The parameter $\mu>0$ denotes the Brinkman coefficient, which plays the role of an effective viscosity. The positive constants $\alpha$ and $\beta$ represent the Darcy coefficient, related to the permeability of the porous medium, and the Forchheimer coefficient, associated with the porosity of the material, respectively.   When $\alpha = \beta = 0$, the system reduces to the classical $d$-dimensional stochastic Navier-Stokes equations (SNSEs). Consequently, system \eqref{1} can be interpreted as the SNSEs with damping. The absorption exponent satisfies $\varpi \in [1,\infty)$, where the case $\varpi = 3$ is referred to as the \emph{critical} exponent, while $\varpi > 3$ corresponds to a rapidly growing nonlinearity (see, for example, \cite{GKM_AMOP}). In the remainder of this work, we distinguish between subcritical, critical, and supercritical SCBFEs corresponding to $\varpi<3$, $\varpi=3$, and $\varpi>3$, respectively.

For spatial dimensions $d=2,3$ and exponents $\varpi \in [3,\infty)$, with the additional assumption $2\mu\beta \geq1$ in the critical case $\varpi=3$, global existence and uniqueness of pathwise strong solutions (in the probabilistic sense) satisfying the energy equality (It\^o’s formula) for system \eqref{1} were established in \cite{KK+MTM-SCBF}. The analysis therein is based on the monotonicity properties of the linear and nonlinear operators combined with a stochastic version of the Minty-Browder argument. For the case of additive rough Gaussian noise taking values in Lebesgue spaces, related results were obtained in \cite{KKMTM-DCDSB}, where the authors also proved the existence of random attractors and invariant measures. More recently, in \cite{KK+FC+MTM-SCBF}, global existence and uniqueness of pathwise strong solutions satisfying the energy equality were proved for system \eqref{1} in all parameter regimes listed in Table~\ref{Table}.
	\begin{table}[ht]
	{	\begin{tabular}{|c|c|c|c|c|}
			\hline
			\textbf{Cases}& $d$ &$ \varpi$& conditions on $\mu>0$, $\alpha>0$ \& $\beta>0$ \\
			\hline
			\textbf{I}& $d=2$ &$1\leq \varpi <\infty$&  no additional condition  \\
			\hline
			\textbf{II}& $d=3$ &$3< \varpi <\infty$& no additional condition \\
			\hline
			\textbf{III}& $d=3$ &$\varpi=3$&for  $2\mu\beta\geq1$ \\
			\hline
		\end{tabular}
		\vskip 0.1 cm
		\caption{Values of $\mu, \alpha, \beta$ and $\varpi$ for the existence of unique strong solution}
		\label{Table}}
\end{table}
 The proof relies on a classical Faedo-Galerkin approximation, combined with stochastic compactness arguments, Jakubowski’s extension of the Skorokhod representation theorem to nonmetric spaces, and the martingale representation theorem. For additional results concerning the deterministic setting and stochastic models driven by L\'evy noise, we refer the reader to \cite{GM_2025,MTM6,MTM9} and the references cited therein.

\subsection{Continuous Data Assimilation} 

Data assimilation aims to incorporate spatially discrete measurements and observational data into physical or numerical models. Its primary motivation stems from the inherent imperfections of deterministic models, which are often incomplete due to the inability to capture all relevant physical processes. In particular, data assimilation has played a central role in improving weather forecasting, and there is growing interest in applying these techniques to a wide range of problems arising in various scientific and engineering applications (\cite{PAM}).  

\subsubsection{Background} In the classical formulation of continuous data assimilation (CDA), as presented by Daley \cite{daley1993atmospheric}, observational measurements are incorporated directly into the model during its time evolution. A more recent development, proposed by Azouani, Olson, and Titi (AOT) in \cite{azouani2014continuous}, recasts CDA as a feedback control (or nudging) procedure inspired by concepts from control theory \cite{AbderrahimTiti2014}. In this approach, an additional relaxation term is introduced to steer the model solution toward the observed data. The resulting assimilated system produces an approximate trajectory that converges asymptotically to the reference (true) solution. Under appropriate assumptions on the regularity of the data and the model, this convergence can be shown to be exponential. Consequently, the AOT algorithm provides an efficient and practical tool for estimating the present and future states of a dynamical system when only sparse observations are available. 
 
 The AOT approach has been extensively employed in the literature for a wide range of deterministic physical models, including equations governing Newtonian and non-Newtonian fluid flows; see, for example, \cite{bessaih2022continuous,FGHMMW,JST,BHLP,PAM,LP,CGJA,BP_SIMA_2021,BB_RMS_2025,TV_DCDS-S_2025,NKC_2025_Submit} and the references therein. Its application to stochastic settings, however, has been comparatively limited (see, e.g., \cite{bessaih2015continuous,Blomker}).  Nevertheless, the inclusion of stochastic effects is crucial for realistic modeling, as physical systems are inherently subject to measurement noise and unresolved fluctuations. More recently, the AOT framework was extended to stochastic systems in \cite{BFLZ_2025_Arxiv}, where it was successfully applied to the two-dimensional stochastic Navier-Stokes equations.

 \subsubsection{Methodology} 
 For completeness, we briefly describe the general structure of the continuous data assimilation algorithm. Assume that $\u=\u(t)$ denotes the unknown true state of a dynamical system, which evolves according to
 \begin{align}
 	\frac{\drm\u(t)}{\drm t} = \Frm (\u (t)),\qquad t>0,
 \end{align}
 with an unknown initial data $\u(0)$. Observational data are assumed to be accessible only at a coarse spatial scale via an interpolant operator $\Rrm_{\h}(\u(t))$, where $\h>0$ represents the observation resolution. The CDA algorithm produces an approximate solution by solving the modified equation
 \begin{align*}
 	\frac{\drm \Urm(t)}{\drm t} = \Frm (\Urm(t)) + \sigma \Rrm_{\h} (\u(t) - \Urm(t)),
 \end{align*}
 where   $\sigma>0$ is the nudging parameter. Provided that the parameters $\sigma$ and $\h$ satisfy suitable conditions, the nudged solution $\Urm(t)$ converges to the true state $\u(t)$ as $t \to +\infty$, regardless of the choice of initial data $\u(0)$ and $\Urm(0)$. The estimates quantifying the difference between $\Urm$ and $\u$ are often called Foias--Prodi estimates, which typically guarantee exponential decay in time when $\sigma$ and $\h$ lie within appropriate bounds. Physically relevant examples of the interpolant operator $\Rrm_{\h}$ can be found in \cite{jones1992determining, jones1993upper, foias1991determining}.

Within the stochastic framework, the dynamics of the system are described by a stochastic differential equation (SDE) given by
\begin{align*}
	\drm \u(t) = \Frm(\u(t)) \, \drm t + \Grm(t, \u(t)) \, \drm \Wrm(t),
\end{align*}
where $\Wrm$ is a $\Qrm$-Wiener process  and $\Grm(\cdot,\cdot)$ is the noise coefficient which may depend on the solution or not.  The data assimilation formulation corresponding to the system is written as
\begin{align*}
	\drm \Urm(t) = \left[ \Frm(\Urm(t)) - \sigma \Rrm_{\h}(\Urm(t) - \u(t)) \right]\, \drm t + \Grm (t,\Urm(t)) \, \drm  \Wrm(t).
\end{align*}
Analyzing convergence in the stochastic setting is more delicate due to the influence of random forcing. One can approach this either \emph{pathwise}, considering almost-sure convergence for individual realizations of the noise, or \emph{in expectation}, focusing on mean-square or statistical convergence. For additive noise, the stochastic contributions cancel when tracking the difference $\u(t)-\Urm(t)$, allowing one to prove exponential convergence along each trajectory as $t \to +\infty$, regardless of the initial states $\u(0)$ and $\Urm(0)$. In contrast, for multiplicative noise, such cancellations do not occur, and convergence is typically established only in the sense of expected values.

 \subsection{Main Results} 
 In this work, we investigate the convergence of the solution $\Urm$ of the nudging equation to the solution $\u$ of system \eqref{1} as time tends to infinity. Due to the distinct challenges presented by the subcritical, critical, and supercritical SCBFEs, we analyze each case separately. First, we establish estimates in expectation for all the cases listed in Table \ref{Table}. For the additive noise scenario, we additionally derive pathwise estimates.  
 
 It is worth to mention that the results presented here are obtained under the assumption of a linear growth condition on the noise coefficient $\Grm$ and for all $\varpi\in(1,\infty)$. Our approach follows the methodology developed in \cite{BFLZ_2025_Arxiv}, where similar analytical tools were employed, particularly those previously used in the study of invariant measures (see, e.g., \cite{GHMR17,KS,BZ_DCDS,FZ25}). We also emphasize that for the two-dimensional case with $\varpi = 1$, results similar to those in \cite{BFLZ_2025_Arxiv} can be obtained; therefore, this particular case is not discussed further in the present work.

 Our main results read as follows. Under the linear growth condition on $\Grm$ and suitable conditions on the observational resolution $\h$ and the nudging parameter $\sigma$, we prove that 
 \begin{enumerate}
 	\item [(i).] for $d=2$ and $1<\varpi<3$ with additive noise (see Theorem \ref{MT-Subcritical} (1))
 	\begin{align*}
 		 \mathbb{E} \left[\|\u(t) - \Urm(t)\|_2^2 \right]\;\to\; 0
 		\quad \text{exponentially fast as } t \to +\infty.
 	\end{align*} 
 \vskip 2mm
 \item  [(ii).]  for $d=2$ and $1<\varpi<3$ with multiplicative noise (see Theorem \ref{MT-Subcritical} (2))
 \begin{align*}
 	  \mathbb{E} \left[\|\u(t) - \Urm(t)\|_2^2 \right] \;\to\; 0
 	\quad \text{$p$-polynomially fast for any } p \in (0,+\infty)
 	\text{ as } t \to +\infty.
 \end{align*} 
\vskip 2mm
\item [(iii).] for $d=2$ and $\varpi=3$ with additive as well as multiplicative noise (see Theorem \ref{MT-Critical} (2))
\begin{align*}
	 \mathbb{E} \left[\|\u(t) - \Urm(t)\|_2^2 \right] \;\to\; 0
	\quad \text{exponentially fast as } t \to +\infty.
\end{align*} 
\vskip 2mm
\item [(iv).] for $d\in\{2,3\}$, $\varpi\geq3$ and $2\beta\mu>1$ with additive as well as multiplicative noise (see Theorems \ref{MT-Critical} (1) and \ref{MT-Supercritical} (2))
\begin{align*}
	 \mathbb{E} \left[\|\u(t) - \Urm(t)\|_2^2 \right] \;\to\; 0
	\quad \text{exponentially fast as } t \to +\infty.
\end{align*} 
\vskip 2mm
\item [(v).] for $d\in\{2,3\}$ and $\varpi>3$ with additive as well as multiplicative noise (see Theorem \ref{MT-Supercritical} (1))
\begin{align*}
	 \mathbb{E} \left[\|\u(t) - \Urm(t)\|_2^2 \right] \;\to\; 0
	\quad \text{exponentially fast as } t \to +\infty.
\end{align*} 
 \end{enumerate}
 
 \begin{remark}
 	We clarify that by exponential convergence, we mean that the difference is bounded by a term of the form $\widetilde{a} e^{-a t}$ for some constants $a>0$ and $\widetilde{a}>0$. Similarly, $p$-polynomial decay refers to a bound of the form $\frac{a}{t^p}$ in the limit, for a suitable constant $a>0$.
 \end{remark}

As we emphasized previously that, we obtain better convergence results than the 2D SNSEs discussed in the work \cite{BFLZ_2025_Arxiv}. The SCBFEs \eqref{1}, also known as SNSEs with damping, exhibit enhanced convergence properties due to the presence of nonlinear damping term. In particular, we show that nonlinear damping not only enables the implementation of CDA in three dimensions but also yields improved convergence results in two dimensions when compared to the classical SNSEs, see the comparison in Table \ref{Table2}. Furthermore, we note that under a suitable sublinear growth condition on the noise coefficient $\Grm$, one can also obtain exponential convergence $\mathbb{E} \left[\|\u(t) - \Urm(t)\|_2^2 \right] \;\to\; 0$ as $t \to +\infty$ even in the case of subcritical SCBFEs (see Remark \ref{Rem-Subcritical} below).

\begin{table}
	\renewcommand{\arraystretch}{1.2}
	\begin{tabular}{|p{3cm}|p{3cm}|p{3cm}|p{3cm}|p{3cm}|}
		\hline
		\textbf{Assumption on the noise coefficient $\Grm$}& 2D SNSEs \cite{BFLZ_2025_Arxiv}  & 	2D Subcritical SCBFEs with $\varpi\neq1$ &  2D and 3D  Critical SCBFEs & 2D and 3D Supercritical SCBFEs   \\
		\hline
		$\Grm$ is independent of $\u$ & exponentially fast & exponentially fast & exponentially fast & exponentially  fast\\
		\hline
		$\Grm$ is bounded & exponentially fast &exponentially  fast& exponentially  fast& exponentially  fast\\
		\hline
		$\Grm$ has sublinear growth & $p$-polynomially fast for any  $p \in (0,+\infty)$ & $p$-polynomially fast for any  $p \in (0,+\infty)$  & exponentially  fast& exponentially  fast\\
		\hline
		$\Grm$ has linear growth & $p$-polynomially fast for any  $p \in (0,\frac{\mu\lambda_1}{4 \widetilde{K}}-\frac12)$ &  $p$-polynomially fast for any  $p \in (0,+\infty)$ &  exponentially  fast& exponentially  fast \\
		\hline
	\end{tabular}
	\vskip 0.1 cm 
	\caption{Comparison of rate of convergence of $ \mathbb{E} \left[\|\u(t) - \Urm(t)\|_2^2 \right] \;\to\; 0$ for SNSEs and SCBFEs}
	\label{Table2}
\end{table}

\subsection{Outline of the Manuscript}
The paper is organized as follows. Section~\ref{sec:setting} presents the necessary preliminary material. In Section~\ref{Sec3}, we discuss the well-posedness of system \eqref{SCBFE} and the associated data assimilation system. Section~\ref{Sec4} is devoted to deriving energy estimates for the solutions of system \eqref{SCBFE} and its data assimilation counterpart, both in expectation and in probability. The continuous data assimilation (CDA) algorithm for 2D subcritical SCBFEs is analyzed in Section~\ref{Sec5} (Theorem \ref{MT-Subcritical}), while Section~\ref{Sec6} addresses the CDA algorithm for 2D and 3D critical SCBFEs (Theorem \ref{MT-Critical}). Section~\ref{Sec7} focuses on the CDA algorithm for 2D and 3D supercritical SCBFEs (Theorem \ref{MT-Supercritical}). Finally, the last section presents a pathwise analysis of the CDA algorithm in the case of additive noise (Theorems \ref{pathwise_data_ass} and \ref{pathwise_data_ass-geq3}).

\section{Mathematical setting}\label{sec:setting}
In this section, we collect the notation, functional spaces, operators, and hypotheses that will be employed in the subsequent analysis.

For any Banach space $E$, we denote its topological dual by $E^*$.   The duality pairing between $E$ and $E^*$ is denoted by $\langle \cdot , \cdot \rangle_{E,E^*}$. When no ambiguity arises, the subscripts will be omitted. For any real Hilbert space $\Hcal$, we denote by $\|\cdot\|_{\Hcal}$ the associated norm and by $(\cdot,\cdot)_{\Hcal}$ the inner product.

\subsection{Functional spaces}
 
Let $\C_0^{\infty}(\Ocal;\R^d)$ denote the space of all infinitely differentiable, $\R^d$-valued functions with compact support in $\Ocal\subset\R^d$. We define
\begin{align*}
	\mathcal{V} &:= \{\u \in \C_0^{\infty}(\Ocal;\R^d) : \nabla \cdot \u = 0\},  &&	\Hbb :=  \overline{\mathcal{V}}^{\Lbb^2(\Ocal) = \mathrm{L}^2(\Ocal;\R^d)},\\
	\Vbb &:= \overline{\mathcal{V}}^{\Hbb^1(\Ocal) = \mathrm{H}^1(\Ocal;\R^d)},
	&& \widetilde{\Lbb}^{p} := \overline{\mathcal{V}}^{\Lbb^p(\Ocal) = \mathrm{L}^p(\Ocal;\R^d)}, \quad p \in (1,\infty), \quad p\neq 2.
\end{align*}

The space $\Hbb$ is equipped with the Hilbert structure induced by $\Lbb^2(\Ocal)$, and its associated norm and inner product are denoted by $\|\cdot\|_2$ and $(\cdot,\cdot)$, respectively.   Owing to the Poincar\'e inequality on the bounded domain $\Ocal$, the space $\Vbb$ can be endowed with the norm $\|\u\|_{\Vbb} := \|\nabla \u\|_2,$ for  $\u \in \Vbb$. The space $\wi\Lbb^p$, $p\in (1,\infty)$ with $p\neq 2$, is equipped with the Banach space structure induced by $\Lbb^p(\Ocal)$, and its associated norm is denoted by $\|\cdot\|_p$.

Following \cite[Subsection~2.1]{RFHK}, for $\frac{1}{p}+\frac{1}{p'}=1$, the sum space $\Vbb^* + \widetilde{\Lbb}^{p'}$ is well defined and forms a Banach space equipped with the norm
\begin{align*}
	\|\u\|_{\widetilde{\Lbb}^{p'}+\Vbb^* } 
	&:= \inf \left\{ \|\u_1\|_{\Vbb^*} + \|\u_2\|_{p'} : 
	\u = \u_1 + \u_2, \; \u_1 \in \Vbb^*, \; \u_2 \in \widetilde{\Lbb}^{p'} \right\} \nonumber\\
	&= \sup \left\{ \frac{|\langle \v ,  \u\rangle|}{
		\|\v\|_{\widetilde{\Lbb}^p \cap \Vbb}} : 
	\boldsymbol{0} \neq \v \in \widetilde{\Lbb}^p \cap \Vbb \right\},
\end{align*}
where $\|\cdot\|_{\widetilde{\Lbb}^p \cap \Vbb} := 
\max\left\{ \|\cdot\|_{p}, \, \|\cdot\|_{\Vbb} \right\}$ defines a norm on the Banach space $\widetilde{\Lbb}^p \cap \Vbb$.  
Moreover, the norm $\max\{\|\u\|_{p}, \|\u\|_{\Vbb}\}$ is equivalent to both 
$\|\u\|_{p} + \|\u\|_{\Vbb}$ and 
$\sqrt{\|\u\|_{p}^2 + \|\u\|_{\Vbb}^2}$ on $\widetilde{\Lbb}^p \cap \Vbb$.

Finally, the following continuous embeddings hold:
\[
\widetilde{\Lbb}^p \cap \Vbb \hookrightarrow \Vbb \hookrightarrow \Hbb \cong \Hbb^* 
\hookrightarrow \Vbb^* \hookrightarrow \widetilde{\Lbb}^{p'} + \Vbb',
\]
with $\Vbb \hookrightarrow \Hbb $ as compact embedding. For a detailed functional framework, we refer the reader to \cite{GKM_AMOP,Temam_1984}, among others.

\subsection{The Helmholtz-Hodge projection}
Let $\mathcal{P}_p : \Lbb^2(\Ocal)\cap \Lbb^p(\Ocal) \to \Hbb\cap\widetilde{\Lbb}^p$ be a bounded linear projection (see \cite{Farwig+Kozono+Sohr_2007,DFHM,HKTY}) such that, for $2\leq p<\infty$, the adjoint map  $(\mathcal{P}_p)^{*}=\mathcal{P}_{p'}$, where $\frac{1}{p}+\frac{1}{p'}=1$ and 
\begin{align*}
	\mathcal{P}_{p'} : \Lbb^2(\Ocal)+ \Lbb^{p'}(\Ocal) \to \Hbb+\widetilde{\Lbb}^{p'}.
\end{align*}
\begin{remark}
	For, $1\leq \varpi <\infty$, the projection $\mathcal{P}_{\frac{\varpi+1}{\varpi}}$ has a continuous linear extension (keeping the same notation) ${\mathcal{P}}_{\frac{\varpi+1}{\varpi}}: \Hbb^{-1} (\Ocal) + \Lbb^{\frac{\varpi+1}{\varpi}} \to \Vbb^* + \wi\Lbb^{\frac{\varpi+1}{\varpi}}$, see e.g. \cite[Proposition 3.1]{Kunstmann_2010}.
\end{remark}

\subsection{Linear operator}\label{opeA}
Let us  introduce  the linear operator defined by 
\begin{equation*}
	\Acal\u:=-\mathcal{P}_{\frac{\varpi+1}{\varpi}}\Delta\u,\ \u\in\Vbb\cap\wi\Lbb^{\varpi+1}.
\end{equation*}
Remember that the operator $\Acal$ is a non-negative operator in $\Hbb$ and 
\begin{align*}
	\left<\u, \Acal \u\right>=\|\nabla\u\|_{2}^2,\ \textrm{ for all }\ \u\in\Vbb\cap\wi\Lbb^{\varpi+1}, \ \text{ so that }\ \|\Acal\u\|_{\Vbb^{*}+\wi\Lbb^{\frac{\varpi+1}{\varpi}}} \leq \|\u\|_{\Vbb}.
\end{align*}
\begin{remark}
	The Stokes operator $\Arm\u:= -\mathcal{P}_{2}\Delta\u$ is a linear operator in $\Hbb$ with domain $\Drm(\Arm)=:\Hbb^2(\Ocal)\cap\Vbb$. It can be extended (and let us not change the notation) as a linear operator $\Arm:\Vbb\to \Vbb^*$ by
	\[
	\langle \u , \Arm\v \rangle:=
	(\nabla \u, \nabla\v)\,,  \text{ for all }  \u,\v\in \Vbb.
	\]
	For the bounded domain $\Ocal$, the operator $\Arm$ is invertible and its inverse $\Arm^{-1}$ is bounded, self-adjoint and compact in $\Hbb$. Thus, using spectral theorem, the spectrum of $\Arm$ consists of an infinite sequence $0< \lambda_1\leq \lambda_2\leq\ldots\leq \lambda_k\leq \ldots,$ with $\lambda_k\to\infty$ as $k\to\infty$ of eigenvalues. Moreover, there exists an orthonormal basis $\{e_k\}_{k=1}^{\infty} $ of $\Hbb$ consisting of eigenfunctions of $\Arm$ such that $\Arm e_k =\lambda_ke_k$,  for all $ k\in\mathbb{N}$. In addition, we have 
		\begin{align}\label{poin}
	\lambda_1\|\u\|_{2}^2 \leq 	\|\u\|_{\Vbb}^2.
	\end{align}
\end{remark}

\subsection{Bilinear operator}
Let us define the \emph{trilinear form} $b(\cdot,\cdot,\cdot):\Vbb\cap\wi\Lbb^{\varpi+1}\times\Vbb\cap\wi\Lbb^{\varpi+1}\times\Vbb\cap\wi\Lbb^{\varpi+1}\to\R$ by $$b(\u,\v,\w)=\int_{\Ocal}(\u(x)\cdot\nabla)\v(x)\cdot\w(x)\drm x = \sum_{i,j=1}^d\int_{\Ocal}\u_i(x)\frac{\partial \v_j(x)}{\partial x_i}\w_j(x)\d x.$$ If $\u, \v$ are such that the linear map $b(\u, \v, \cdot) $ is continuous on $\Vbb\cap\wi\Lbb^{\varpi+1}$, the corresponding element of $\Vbb^*+\wi\Lbb^{\frac{\varpi+1}{\varpi}}$ is denoted by $\B(\u, \v)$. We also denote  $\B(\u, \u)= \mathcal{P}_{\frac{\varpi+1}{\varpi}} [(\u\cdot\nabla)\u]$.
An integration by parts yields  
\begin{equation}\label{b0}
	\left\{
	\begin{aligned}
		b(\u,\v,\v) &= 0, && \text{ for all }\ \u,\v \in\Vbb,\\
		b(\u,\v,\w) &=  -b(\u,\w,\v), && \text{ for all }\ \u,\v,\w\in \Vbb.
	\end{aligned}
	\right.\end{equation}

   \begin{remark}
   	 1).  For $\varpi>3$, using H\"older's and Young's inequalities, we estimate $|b(\u,\v,\w)|$  as  (see \cite[Equation (5.21)]{KK+FC+MTM-SCBF})
   	\begin{align}\label{bilinear-etimate-1}
   		|b(\u,\v,\w)|  & \leq \|\v\|_{\Vbb}\||\w|\u\|_2   \leq \frac{\mu}{2} \|\nabla\v\|^2_2 + \frac{1}{2\mu} \||\w|\u\|_2^2
   		\nonumber  \\
   		& \leq \frac{\mu}{2} \|\v\|^2_{\Vbb} + \frac{\beta}{4}  \||\w|^{\frac{\varpi -1}{2}}\u\|_2^2  +  \widehat{\boldsymbol{\upvarpi}}  \|\u\|_2^2.
   	\end{align}
   	where 
   	\begin{align}\label{eqn-upvarpi}
   		\widehat{\boldsymbol{\upvarpi}} =	\frac{\varpi-3}{2\mu(\varpi-1)}\left(\frac{4}{\mu\beta (\varpi-1)}\right)^{\frac{2}{\varpi-3}}.
   	\end{align}

   	2). Also, as in \cite{Zhou_2012}, for $r> 3$, one can estimate $|b(\u,\v,\w)|$ as (see also \cite[Remark 2.7]{GM_2025})
   	\begin{align}\label{bilinear-etimate-2}
   		|b(\u,\v,\w)|  & \leq  \frac{1}{2\beta} \|\v\|_{\Vbb}^2+\frac{\beta}{2} \||\w|^{\frac{\varpi-1}{2}}\u\|_{2}^2 + \frac{\beta}{2}\|\u\|_{2}^2.
   	\end{align}
   	
   \end{remark}

    \subsection{Nonlinear operator}
We introduce the  nonlinear operator defined by
\begin{equation*}
	\K(\u):= \Pcal_{\frac{\varpi+1}{\varpi}}[|\u|^{\varpi-1}\u],\ \u\in\Vbb\cap\wi\Lbb^{\varpi+1}.
\end{equation*}
It can be easily seen that $\langle\u, \K (\u)\rangle = \|\u\|_{\widetilde{\Lbb}^{\varpi+1}}^{\varpi+1}$. An application of the Mean Value Theorem and the H\"older inequality provides the estimate (see \cite[Subsection 2.4, p. 8]{GM_2025})
\begin{align*}
	&\langle\w, \K(\u) - \K(\v)\rangle \leq \varpi\left(\|\u\|_{\widetilde{\Lbb}^{\varpi+1}}+\|\v\|_{\widetilde{\Lbb}^{\varpi+1}}\right)^{\varpi-1}\|\u-\v\|_{\widetilde{\Lbb}^{\varpi+1}}\|\w\|_{\widetilde{\Lbb}^{\varpi+1}},
\end{align*}
for all $\u,\v, \w\in\Vbb\cap\widetilde{\Lbb}^{\varpi+1}$. 
Thus the operator $\K(\cdot):\Vbb\cap\widetilde{\Lbb}^{\varpi+1}\to \Vbb^{*}+\widetilde{\Lbb}^{\frac{\varpi+1}{\varpi}}$ is locally Lipschitz. Furthermore, for any $\varpi\in[1,\infty)$, we have (see \cite{KK+MTM-SCBF})
\begin{align}\label{2.23}
	&\langle\u-\v, \K(\u)-\K(\v)\rangle\geq \frac{1}{2}\||\u|^{\frac{\varpi-1}{2}}(\u-\v)\|_{2}^2+\frac{1}{2}\||\v|^{\frac{\varpi-1}{2}}(\u-\v)\|_{2}^2 \geq \frac{1}{2^{\varpi-1}}\|\u-\v\|_{\wi\Lbb^{\varpi+1}}^{\varpi+1}\geq 0,
\end{align}
for $\varpi\geq 1$ 	and all $\u,\v\in\Vbb\cap\wi\Lbb^{\varpi+1}$.

\subsection{Stochastic setting}\label{Sto-set}
Let $(\Omega,\mathscr{F},\mathbb{P})$ be a complete probability space endowed with an increasing filtration $\{\mathscr{F}_t\}_{t\geq 0}$ of sub-$\sigma$-algebras of $\mathscr{F}$ satisfying the usual conditions.  

\subsubsection{$\Qrm$-Wiener process} 
We begin by recalling the definition and some properties of $\Qrm$-Wiener processes. Let $\Hcal$ be a separable Hilbert space.

\begin{definition}
	A stochastic process $\{\Wrm(t)\}_{t\geq 0}$ is called an \emph{$\Hcal$-valued, $\mathscr{F}_t$-adapted $\Qrm$-Wiener process} with covariance operator $\Qrm$ if:
	\begin{enumerate}
		\item[(i)] for every non-zero $\v\in \Hcal$, the process $\|\Qrm^{1/2}\v\|_{\Hcal}^{-1}(\Wrm(t),\v)$ is a standard one-dimensional Wiener process;
		\item[(ii)] for any $\v\in \Hcal$, the real-valued process $(\Wrm(\cdot),\v)$ is an $\mathscr{F}_t$-adapted martingale.
	\end{enumerate}
\end{definition}

The process $\{\Wrm(t)\}_{t\ge 0}$ is a $\Qrm$-Wiener process if and only if, for each $t$, one can represent the random field $\Wrm(\cdot)$ as 
$$\Wrm(\cdot,x) = \sum_{k=1}^{\infty} \sqrt{\mu_k}\,\boldsymbol{q}_k(x)\beta_k(\cdot),$$
where $\{\beta_k(\cdot)\}_{k\in\mathbb{N}}$ are independent one-dimensional Brownian motions on $(\Omega,\mathscr{F},\mathbb{P})$, and $\{\boldsymbol{q}_k\}_{k=1}^{\infty}$ is an orthonormal basis of $\Hcal$ satisfying $\Qrm\boldsymbol{q}_k = \mu_k \boldsymbol{q}_k$. If, in addition, $\Tr\Qrm = \sum_{k=1}^{\infty}\mu_k < \infty$, then $\Wrm(\cdot)$ is a Gaussian process in $\Hcal$ with 
$$\Ebb[\Wrm(t)] = 0, \qquad \mathrm{Cov}[\Wrm(t)] = t\Qrm,\qquad t\ge 0.$$

Define $\Hcal_0 := \Qrm^{1/2}\Hcal$, which becomes a Hilbert space when equipped with the inner product
$$(\u,\v)_0 = \sum_{k=1}^{\infty}\frac{1}{\mu_k}(\u,\boldsymbol{q}_k)(\v,\boldsymbol{q}_k)
= (\Qrm^{-1/2}\u,\Qrm^{-1/2}\v), \qquad \u,\v\in \Hcal_0,$$
where $\Qrm^{-1/2}$ denotes the pseudo-inverse of $\Qrm^{1/2}$.

Let $\mathcal{L}(\Hcal)$ be the space of bounded linear operators on $\Hcal$, and let $\mathcal{L}_{\Qrm}:=\mathcal{L}_{\Qrm}(\Hcal)$ denote the space of Hilbert-Schmidt operators from $\Hcal_0$ into $\Hcal$. Since $\Qrm$ is trace-class, the embedding $\Hcal_0 \hookrightarrow \Hcal$ is Hilbert–Schmidt, and $\mathcal{L}_{\Qrm}$ is itself a Hilbert space with norm
$$\|\Phi\|_{\mathcal{L}_{\Qrm}}^2 = \Tr(\Phi\Qrm\Phi^*) 
= \sum_{k=1}^{\infty} \|\Qrm^{1/2}\Phi^*\boldsymbol{q}_k\|_{\Hcal}^2,$$
and inner product
$$(\Phi,\Psi)_{\mathcal{L}_{\Qrm}} 
= \Tr(\Phi\Qrm\Psi^*) 
= \sum_{k=1}^{\infty} \bigl(\Qrm^{1/2}\Psi^*\boldsymbol{q}_k,\Qrm^{1/2}\Phi^*\boldsymbol{q}_k\bigr).$$
Further background can be found in \cite{DaZ}.

We now state the assumptions imposed on the noise coefficient $\Grm$.

\begin{hypothesis}\label{hyp-noise}
	Assume that $\{\Wrm(t)\}_{t\ge 0}$ is an $\Hbb$-valued $\Qrm$-Wiener process defined on the stochastic basis $(\Omega,\mathscr{F},\{\mathscr{F}_t\}_{t\geq 0},\mathbb{P})$. The coefficient $\Grm(\cdot,\cdot)$ satisfies:
	\begin{itemize}
		\item[(H.1)] $\Grm \in \C([0,T]\times (\Vbb\cap\wi\Lbb^{\varpi+1});\,\mathcal{L}_{\Qrm}(\Hbb))$;
		
		\item[(H.2)] \textbf{(Linear growth condition)}  
		There exist two constants $K, \widetilde{K} > 0$ such that for all $t\in [0,T]$ and $\u\in\Hbb$,
\begin{align*}
	\|\Grm(t,\u)\|_{\mathcal{L}_{\Qrm}}^2 	\leq K+\widetilde{K}\|\u\|_{2}^2;
\end{align*}
		
		\item[(H.3)] \textbf{(Lipschitz condition)}  
		There exists $L\geq 0$ such that for any $t\in[0,T]$ and $\u_1,\u_2\in\Hbb$,
\begin{align*}
		\|\Grm(t,\u_1) - \Grm(t,\u_2)\|_{\mathcal{L}_{\Qrm}}^2 \leq L\,\|\u_1 - \u_2\|_{2}^2.
\end{align*}
	\end{itemize}
\end{hypothesis}

\subsection{Interpolant operators}\label{sec-inter}
In the absence of knowledge of the initial velocity $\u_0$, the data assimilation framework for the 2D and 3D stochastic convective Brinkman-Forchheimer equations incorporates a linear interpolant operator $\Rrm_{\h} : \mathbb{H}^1(\Ocal) \to \mathbb{L}^2(\Ocal),$ which interpolates a given function at a spatial resolution of length scale $\h>0$ and satisfies the following approximation property:
 \begin{equation}\label{eqn-data-inter}
\left\| \u -  \Rrm_{\h}(\u) \right\|_{2}^2 \leq c_0 \h^2 \left\|\u \right\|_{\mathbb{H}^1}^2,
\end{equation}
 for every $\u \in \mathbb{H}^1(\Ocal).$ 
\begin{example}
	1). An example of an interpolant which is physically relevant and satisfies \eqref{eqn-data-inter} is given by finite volume elements  (see \cite{azouani2014continuous} and the therein references). 
	Specifically, let $\h > 0$ be given, and let 
	\[
	\Ocal = \bigcup_{j=1}^{N_{\h}} \Ocal_j,
	\]
	where the $\Ocal_j$ are disjoint subsets such that $\operatorname{diam}(\Ocal_j) \leq \h$ for 
	$j = 1, 2, \ldots, N_{\h}$. Then we set
	\[
	\Rrm_{\h}(\u)(x) = \sum_{j=1}^{N_{\h}} \overline{\u}_j\,\pmb{1}_{\Ocal_j}(x),
	\]
	where 
	\[
	\overline{\u}_j=\frac{1}{|\Ocal_j|}\int_{\Ocal_j} \u (x) \drm x.
	\]
	
	2). The orthogonal projection onto the Fourier modes, with wave numbers $k$ such that $|k| \leq \frac{1}{\h},$ is an another example of an interpolant operator which satisfies  the approximation property \eqref{eqn-data-inter}.
\end{example}

\section{Underlying system and data assimilation system}\label{Sec3}
In this section, we discuss the solvability of the underlying system and introduce the data assimilation system, along with a discussion of its solvability. Taking the projection $\Pcal_{\frac{\varpi+1}{\varpi}}$ of system \eqref{1}, we write
\begin{equation}\label{SCBFE}
	\left\{
	\begin{aligned}
		\drm \u + \mu \Acal\u \drm t + \B(\u,\u) \drm t + \alpha \u \drm t + \beta\K(\u) \drm t  & =  \f \drm t + \Grm(t,\u)\drm\Wrm, &&t>0, \\ 
		\u(0)&=\u_0. && 
	\end{aligned}
	\right.
\end{equation} 
In what follows, we recall the solvability result for system \eqref{SCBFE} established in \cite{KK+FC+MTM-SCBF}.
\begin{definition}\label{def-StrongSolution}
	We say that the stochastic system \eqref{SCBFE} admits a \emph{strong solution} (in the probabilistic sense) if, for every stochastic basis
\begin{align*}
	(\Omega,\mathscr{F},\{\mathscr{F}_t\}_{t\geq 0},\Pbb)
\end{align*}
	and every $\Qrm$-Wiener process $\{\Wrm(t)\}_{t\geq 0}$ defined on this basis, there exists a progressively measurable process
\begin{align*}
		\u : [0,+\infty)\times\Omega \to \Hbb
\end{align*}
	such that, $\Pbb$-a.s.,
\begin{align*}
		\u(\cdot,\omega)\in \C([0,+\infty);\Hbb_w)\cap \Lrm_{\mathrm{loc}}^2(0,+\infty;\Vbb)\cap \Lrm_{\mathrm{loc}}^{\varpi+1}(0,+\infty;\wi\Lbb^{\varpi+1}),
\end{align*}
	and for all $t\in[0,T]$ and all $\v\in\Vbb\cap\wi\Lbb^{\varpi+1}$, the following identity holds $\Pbb$-a.s.:
	\begin{align*}
	& (\u(t),\v) 
	\nonumber\\ & =(\u_0,\v)-\int_0^t\langle\v , \; \mu \Acal\u(s)+\B(\u(s),\u(s)) + \alpha \u(s) +\beta\K(\u(s)) - \f\rangle\d s 
	  + \int_0^t\left(\v , \; \Grm(s,\u(s))\d\W(s)\right).
\end{align*}
\end{definition}

 \begin{theorem}[{\cite[Theorem 3.10]{KK+FC+MTM-SCBF}}]\label{thm-StrongSolution}
	Under  Hypothesis \ref{hyp-noise}, let $\f \in \Vbb^*$ and $\u_0 \in \Hbb$. 
	Then, for all the cases given in Table \ref{Table}, the system \eqref{SCBFE} admits a unique strong solution $\u$ in the sense of Definition \ref{def-StrongSolution} with $\Pbb$-a.s. paths in 
	\begin{equation*}
		\C([0,+\infty), \Hbb) \cap \Lrm_{\mathrm{loc}}^2(0,+\infty;\Vbb)\cap \Lrm_{\mathrm{loc}}^{\varpi+1}(0,+\infty;\wi\Lbb^{\varpi+1}).
	\end{equation*}
\end{theorem}

We now introduce the data assimilation equation. Let $\u$ denote the solution of the stochastic convective Brinkman-Forchheimer equations \eqref{SCBFE}, and define the process $\Urm$ as the solution of the following equation, which we refer to as the data assimilation equation:
\begin{equation}\label{SCBFE-CDA}
	\left\{
	\begin{aligned}
		\drm \Urm + \mu \Acal\Urm \drm t + \B(\Urm,\Urm) \drm t + \alpha \Urm \drm t + \beta\K(\Urm) \drm t  & =  \f \drm t + \Grm(t,\Urm)\drm\Wrm - \sigma \Pcal_{\frac{\varpi+1}{\varpi}} \Rrm_{\h}\left(\Urm-\u\right)  \drm t , &&t>0, \\ 
		\Urm(0)&=\Urm_0, && 
	\end{aligned}
	\right.
\end{equation} 
where $\u$ is the solution of \eqref{SCBFE} in the sense of Definition \ref{def-StrongSolution} and $\Rrm_{\h}$ defined in Subsection \ref{sec-inter} which satisfies \eqref{eqn-data-inter}.

We observe that, if $\Rrm_{\h}$ satisfies condition \eqref{eqn-data-inter}, then stochastic system \eqref{SCBFE-CDA} admits a unique strong solution (in the probabilistic sense) with the same regularity as the solution of system \eqref{SCBFE}. Indeed, the additional nudging term $\sigma \Pcal_{\frac{\varpi+1}{\varpi}} \Rrm_{\h}(\Urm-\u)$ introduced by the continuous data assimilation procedure can be controlled using the same a priori estimates as in the original system, and therefore does not affect the well-posedness analysis. Consequently, we state the following result:
\begin{theorem}
	Under  Hypothesis \ref{hyp-noise}, let $\f \in \Vbb^*$ and $\u_0 \in \Hbb$. 
	Then, for all the cases given in Table \ref{Table}, the system \eqref{SCBFE-CDA} admits a unique strong solution $\Urm$ with $\Pbb$-a.s. paths in 
	\begin{equation*}
		\C([0,+\infty), \Hbb) \cap \Lrm_{\mathrm{loc}}^2(0,+\infty;\Vbb)\cap \Lrm_{\mathrm{loc}}^{\varpi+1}(0,+\infty;\wi\Lbb^{\varpi+1}),
	\end{equation*}
and for all $t\in[0,T]$ and all $\v\in\Vbb\cap\wi\Lbb^{\varpi+1}$, the following identity holds $\Pbb$-a.s.:
\begin{align*}
	 (\Urm(t),\v) 
	& =(\Urm_0,\v)-\int_0^t\langle\v , \; \mu \Acal\Urm(s)+\B(\Urm(s),\Urm(s)) + \alpha \Urm(s) +\beta\K(\Urm(s)) - \f\rangle\d s 
	\nonumber\\ 
	& \quad  + \int_0^t\left(\v , \; \Grm(s,\Urm(s))\d\W(s)\right) 
	+ \sigma \int_0^t\left(\v , \Rrm_{\h}(\Urm(s)-\u(s))\right)\drm s.
\end{align*}
\end{theorem}

\section{Estimates in expectation and probability}\label{Sec4}
In this section, we develop a collection of estimates in expectation and probability that are fundamental to the derivation of the main results of this article.

\subsection{Estimates in expectation}
We first provide estimates in the expectation which will be useful in the sequel.

\begin{lemma}\label{lem-pth-moments}
Under Hypothesis \ref{hyp-noise}, for $p\geq 1$, $\varpi>1$, there exists a constant $\Mrm := \Mrm_{p,\mu,\beta,\varpi,K,\widetilde{K},|\Ocal|,\|\f\|_{\Vbb^{*}}}>0$ depends on $p$, $\mu$, $\beta$, $\varpi$, $K$, $\widetilde{K}$, $|\Ocal|$ (measure of the bounded domain $\Ocal$) and $\|\f\|_{\Vbb^{*}}$, and  independent of $t$ such that the solution $\u$ of \eqref{SCBFE} satisfies
	\begin{align}\label{eqn-moments-2p-u}
		& \sup_{t\geq 0}	\Ebb \left[\|\u(t)\|^{2p}_{2}\right] \leq \|\u_0\|^{2p}_{2} + \frac{\Mrm}{p (\mu\lambda_1+2\alpha)}.
	\end{align}
\end{lemma}
\begin{proof}
	An application of the It\^o formula to the process $\|\u(\cdot)\|^{2p}_{2}$ ($p\geq1$), and H\"older's and Young's inequalities yield
	\begin{align}
		&\|\u(t)\|^{2p}_{2}+ 2p\int_{0}^{t}\|\u(s)\|_{2}^{2p-2}\left(\mu\|\u(s)\|^2_{\Vbb}+\alpha\|\u(s)\|^2_{2}+\beta\|\u(s)\|^{\varpi+1}_{\varpi+1}\right)\drm s
		\nonumber\\&
		= \|\u_0\|^{2p}_{2} +  p \int_{0}^{t}\|\u( s)\|_{2}^{2p-2}\|\Grm(s,\u(s))\|^2_{\mathcal{L}_{\Qrm}} \drm s
		+  2 p\int_{0}^{t}\|\u( s)\|_{2}^{2p-2}\left(\Grm(s,\u(s))\drm \Wrm(s) , \u(s)\right) 
		\nonumber\\ & \quad  +  2 p\int_{0}^{t}\|\u( s)\|_{2}^{2p-2}\left< \u( s),\f  \right> \drm s
		 + 2 p(p-1) \int_{0}^{t}\|\u( s)\|^{2p-4}_{2} \|(\Grm(s,\u(s)))^{*}\u( s)\|^2_{2}\drm s \label{eqn-forProbabilityEstimates}
		\\&
		\leq \|\u_0\|^{2p}_{2} +  p(2p-1) \int_{0}^{t}\|\u( s)\|_{2}^{2p-2}\|\Grm(s,\u(s))\|^2_{\mathcal{L}_{\Qrm}} \drm s
		 +  2 p\int_{0}^{t}\|\u( s)\|_{2}^{2p-2}\|\f\|_{\Vbb^{*}} \|\u( s)\|_{\Vbb} \drm s 
		\nonumber\\ 
		&\quad +  2 p\int_{0}^{t}\|\u( s)\|_{2}^{2p-2}\left(\Grm(s,\u(s))\drm \Wrm(s) , \u( s)\right) 
		\nonumber\\&
		\leq \|\u_0\|^{2p}_{2} +  p(2p-1)K \int_{0}^{t}\|\u( s)\|_{2}^{2p-2} \drm s 
		+  p(2p-1)\widetilde{K} \int_{0}^{t}\|\u( s)\|_{2}^{2p}\drm s
		+  2 p\int_{0}^{t}\|\u( s)\|_{2}^{2p-2}\|\f\|_{\Vbb^{*}} \|\u( s)\|_{\Vbb} \drm s 
		\nonumber\\ 
		&\quad +  2 p\int_{0}^{t}\|\u( s)\|_{2}^{2p-2}\left(\Grm(s,\u(s))\drm \Wrm(s) , \u( s)\right) 
		\nonumber\\ 
		& \leq \|\u_0\|^{2p}_{2}  +    p \mu \int_{0}^{t}\|\u(s)\|_{2}^{2p-2} \|\u(s)\|^2_{\Vbb} \drm s   + p \beta \int_{0}^{t}\|\u(s)\|_{2}^{2p-2} \|\u(s)\|^{\varpi+1}_{\varpi+1} \drm s  +  t \Mrm_{p,\mu,\beta,\varpi,K,\widetilde{K},|\Ocal|,\|\f\|_{\Vbb^{*}}} 
		\nonumber\\ 
		&\quad +  2 p\int_{0}^{t}\|\u( s)\|_{2}^{2p-2}\left(\Grm(s,\u(s))\drm \Wrm(s) , \u( s)\right),\label{EE1}
	\end{align}
where the constant $\Mrm := \Mrm_{p,\mu,\beta,\varpi,K,\widetilde{K},|\Ocal|,\|\f\|_{\Vbb^{*}}} >0$ depends on $p$, $\mu$, $\beta$, $\varpi$, $K$, $\widetilde{K}$, $|\Ocal|$ (measure of the bounded domain $\Ocal$) and $\|\f\|_{\Vbb^{*}}$.

Taking expectations on both sides of \eqref{EE1}, using \eqref{poin}, and noting that the stochastic integral is a (local) martingale, yield
\begin{align*}
	&\Ebb \left[\|\u(t)\|^{2p}_{2}\right]
	+ p (\mu\lambda_1+2\alpha) \int_{0}^{t}\Ebb\left[\|\u(s)\|_{2}^{2p}\right]\drm s 
	  \leq \|\u_0\|^{2p}_{2}  +      t \Mrm_{p,\mu,\beta,\varpi,K,\widetilde{K},|\Ocal|,\|\f\|_{\Vbb^{*}}},
\end{align*}
which, by an application of Gronwall's inequality, implies
\begin{align*}
	\Ebb \left[\|\u(t)\|^{2p}_{2}\right] \leq \|\u_0\|^{2p}_{2} e^{-p (\mu\lambda_1+2\alpha) t} + \frac{\Mrm_{p,\mu,\beta,\varpi,K,\widetilde{K},|\Ocal|,\|\f\|_{\Vbb^{*}}}}{p (\mu\lambda_1+2\alpha)} (1- e^{-p (\mu\lambda_1+2\alpha) t}),
\end{align*}
which immediately gives \eqref{eqn-moments-2p-u}.
\end{proof}

\begin{lemma}\label{lem-pth-moments-Urm}
	Under  Hypothesis \ref{hyp-noise}, for $p\geq 1$, $\varpi>1$, there exists a constant $\widetilde{\Mrm}$ depends on $p$, $\mu$, $\lambda_1$, $\alpha$, $\beta$, $\varpi$, $K$, $\widetilde{K}$, $|\Ocal|$ (measure of the bounded domain $\Ocal$) and $\|\f\|_{\Vbb^{*}}$, and  independent of $t$ such that the solution $\Urm$ of \eqref{SCBFE-CDA} satisfies
	\begin{align}\label{eqn-moments-2p-Urm}
		& \sup_{t\geq 0}	\Ebb \left[\|\Urm(t)\|^{2p}_{2}\right] \leq   \widetilde{\Mrm}(1+\|\u_0\|^{2p}_{2} + \|\Urm_0\|^{2p}_{2}).
	\end{align}
\end{lemma}
\begin{proof}
	Note that system \eqref{SCBFE-CDA} contains an additional term, $-\sigma \Rrm_{\h}(\Urm-\u)$, compared to system \eqref{SCBFE}. This term can be estimated in the same way as in the proof of \cite[Lemma A.4]{BFLZ_2025_Arxiv}. Therefore, we omit the proof. Making use of Lemma \ref{lem-pth-moments}, one can complete the proof.
\end{proof}

\subsection{Estimates in probability}
Let us now provide estimates in the probability depending on exponent $\varpi>1$ which will be useful in the sequel.

\begin{lemma}\label{lem-prob-esti-1}
Suppose that $1< \varpi < +\infty$.	Under  Hypothesis \ref{hyp-noise}, the solution $\u$ of \eqref{SCBFE} satisfies the following estimate:	
	\begin{itemize}
		\item [1.]  For $L=0$,
		\begin{align}\label{eqn-Prob-Est-varpi1}
			& \Pbb\left\{\sup_{t \geq T} \left( \|\u(t)\|^{2}_{2} + \int_{0}^{t}\left(\frac{\mu}{2}\|\u(s)\|^2_{\Vbb}+2\beta\|\u(s)\|^{\varpi+1}_{\varpi+1}\right)\drm s - \|\u_0\|^{2}_{2}  - \left[K  + \frac{\|\f\|_{\Vbb^*}^2 }{\mu} \right]t  \right) \geq R \right\}
			\nonumber\\ 
			& \leq e^{- \frac{1}{8K} \left(\frac{\mu\lambda_1}{2} + 2\alpha\right) R};
		\end{align}
	for all $T\geq 0$ and $R>0$.\\
		\item [2.] For all $2< q <\infty$ and for $L>0$, there exist constants $\widehat{\Mrm}:= \widehat{\Mrm}_{K,\mu,\beta,\varpi,L,|\Ocal|,\|\f\|_{\Vbb^*}}>0$ and $C_{q,K,L}>0$ such that 
		\begin{align}\label{eqn-Prob-Est-varpi2}
			& \Pbb\left\{\sup_{t \geq T} \left( \|\u(t)\|^{2}_{2} + \int_{0}^{t}\left(\mu\|\u(s)\|^2_{\Vbb}+2\alpha\|\u(s)\|^2_{2}+2\beta\|\u(s)\|^{\varpi+1}_{\varpi+1}\right)\drm s - \|\u_0\|_2^2 - \widehat{\Mrm}t -t-2  \right) \geq R \right\}
			\nonumber\\
			& \leq C_{q,K,L} \left[1+ \|\u_0\|^{2q}_{2} + \frac{\Mrm}{q (\mu\lambda_1+2\alpha)}\right]  \frac{1}{(R+T)^{\frac{q}{2}-1}},
		\end{align} 	
		for all $T\geq 0$ and $R>0$, where $\Mrm>0$ is the constant appearing in \eqref{eqn-moments-2p-u}.
	\end{itemize}
The explicit value of the constant $\widehat{\Mrm}$ is as follows:
\begin{align}\label{eqn-widehatM}
	\widehat{\Mrm} = K  + \frac{\|\f\|_{\Vbb^*}^2 }{\mu} + \left\{\frac{\beta(\varpi+1)}{2}\right\}^{-\frac{2}{\varpi-1}} \left\{\frac{\varpi+1}{\varpi-1}\right\}^{-1} \left\{L \right\}^{\frac{\varpi+1}{\varpi-1}}  |\Ocal|.
\end{align}
\end{lemma}
\begin{proof}	
	Using \eqref{eqn-forProbabilityEstimates} for $p=1$, Hypothesis \ref{hyp-noise}, and Cauchy-Schwartz and Young inequalities, we have 
	\begin{align}\label{PrEs1}
		& \|\u(t)\|^{2}_{2} + \int_{0}^{t}\left(\mu\|\u(s)\|^2_{\Vbb}+2\alpha\|\u(s)\|^2_{2}+2\beta\|\u(s)\|^{\varpi+1}_{\varpi+1}\right)\drm s
		\nonumber\\&
		= \|\u_0\|^{2}_{2} -  \mu \int_{0}^{t}\|\u(s)\|^2_{\Vbb} \drm s -  \beta \int_{0}^{t}\|\u(s)\|^{\varpi+1}_{\varpi+1} \drm s + \int_{0}^{t} \|\Grm(s,\u(s))\|^2_{\mathcal{L}_{\Qrm}} \drm s
		+  2 \int_{0}^{t} \left< \u( s),\f  \right> \drm s
		\nonumber\\ & \quad +  \underbrace{2  \int_{0}^{t} \left(\Grm(s,\u(s))\drm \Wrm(s) , \u(s)\right) }_{=: \mathcal{M}(t)}
		\nonumber\\&
		\leq  \|\u_0\|^{2}_{2} -  \mu \int_{0}^{t}\|\u(s)\|^2_{\Vbb} \drm s -  \beta \int_{0}^{t}\|\u(s)\|^{\varpi+1}_{\varpi+1} \drm s + \int_{0}^{t} \|\Grm(s,\u(s))- \Grm(s,\boldsymbol{0})\|^2_{\mathcal{L}_{\Qrm}} \drm s  + \int_{0}^{t} \|\Grm(s,\boldsymbol{0})\|^2_{\mathcal{L}_{\Qrm}} \drm s
		\nonumber\\ 
		&	\quad  +  2 \int_{0}^{t} \left< \u( s),\f  \right> \drm s
		+  \mathcal{M}(t)
		\nonumber\\ 
		& \leq \|\u_0\|^{2}_{2} -  \beta \int_{0}^{t}\|\u(s)\|^{\varpi+1}_{\varpi+1} \drm s + \left[K  + \frac{\|\f\|_{\Vbb^*}^2 }{\mu}\right]t  + L \int_{0}^{t} \|\u(s)\|^2_{2} \drm s
		+   \mathcal{M}(t)
		\nonumber\\ 
		& \leq \|\u_0\|^{2}_{2}  -  \beta \int_{0}^{t}\|\u(s)\|^{\varpi+1}_{\varpi+1} \drm s  + \left[K  + \frac{\|\f\|_{\Vbb^*}^2 }{\mu}\right]t  + L |\Ocal|^{\frac{\varpi-1}{\varpi+1}} \int_{0}^{t} \|\u(s)\|^2_{\varpi+1} \drm s + \mathcal{M}(t)
		\nonumber\\ 
		& \leq \|\u_0\|^{2}_{2}  + \widehat{\Mrm} t  	+   \mathcal{M}(t).
	\end{align}
	where $\widehat{\Mrm}$ is given by \eqref{eqn-widehatM}. Note that $\{\mathcal{M}(t)\}_{t\geq0}$	 is a martingale process with quadratic variation given by
	\begin{align}\label{PrEs2}
		[\mathcal{M}](t) & \leq 4 \int_{0}^{t} \|\u(s)\|_2^2 \|\Grm(s,\u(s))\|^2_{\mathcal{L}_{\Qrm}} \drm s  
		\nonumber\\ 
		& \leq 8 \int_{0}^{t} \|\u(s)\|_2^2 [\|\Grm(s,\u(s)) - \Grm(s,\boldsymbol{0})\|^2_{\mathcal{L}_{\Qrm}} + \|\Grm(s,\boldsymbol{0})\|^2_{\mathcal{L}_{\Qrm}}] \drm s  
		\nonumber\\ 
		& \leq 8 \int_{0}^{t}  \left[L\|\u(s)\|_2^4  + K\|\u(s)\|_2^2 \right] \drm s.
	\end{align}
	
	\vskip 2mm
	\noindent
	\textbf{Case 1.} \textit{When $L=0$}. In this case, we have 
	\begin{align*}
		& \|\u(t)\|^{2}_{2} + \int_{0}^{t}\left(\frac{\mu}{2}\|\u(s)\|^2_{\Vbb}+2\beta\|\u(s)\|^{\varpi+1}_{\varpi+1}\right)\drm s  - \|\u_0\|^{2}_{2}  - \left[K  + \frac{\|\f\|_{\Vbb^*}^2 }{\mu} \right]t
		\nonumber\\ 
		&  \leq \mathcal{M}(t)  - \int_{0}^{t}\left(\frac{\mu}{2}\|\u(s)\|^2_{\Vbb}+\alpha\|\u(s)\|^2_{2}\right) \drm s
		\nonumber\\ 
		&  \leq \mathcal{M}(t)  - \left(\frac{\mu\lambda_1}{2} + 2\alpha\right) \int_{0}^{t}\|\u(s)\|^2_{2} \drm s
		\nonumber\\ 
		&  \leq \mathcal{M}(t)  - \frac{1}{8K} \left(\frac{\mu\lambda_1}{2} + 2\alpha\right) [\mathcal{M}](t),
	\end{align*}
	which implies by the exponential martingale inequality \cite[Proposition 3.1]{Glatt-Holtz_2014_Arxiv} that 
	\begin{align*}
		& \Pbb\left\{\sup_{t \geq T} \left( \|\u(t)\|^{2}_{2} + \int_{0}^{t}\left(\frac{\mu}{2}\|\u(s)\|^2_{\Vbb}+2\beta\|\u(s)\|^{\varpi+1}_{\varpi+1}\right)\drm s - \|\u_0\|^{2}_{2}  - \left[K  + \frac{\|\f\|_{\Vbb^*}^2 }{\mu} \right]t  \right) \geq R \right\}
		\nonumber\\ 
		&  \leq \Pbb \left\{  \sup_{t \geq T} \left( \mathcal{M}(t)  - \frac{1}{8K} \left(\frac{\mu\lambda_1}{2} + 2\alpha\right) [\mathcal{M}](t) \right) \geq R \right\}  \leq e^{- \frac{1}{8K} \left(\frac{\mu\lambda_1}{2} + 2\alpha\right) R}.
	\end{align*}
	This completes the proof of \textbf{Case 1.}

	\vskip 2mm
	\noindent
	\textbf{Case 2.} \textit{When $L>0$}. Let us subtract $t+2$ from both side of \eqref{PrEs1} and apply \cite[Lemma B.7]{DSZ_2025_Arxiv} to obtain for any $T\geq 0$ and $q>2$
	\begin{align*}
		& \Pbb\left\{\sup_{t \geq T} \left( \|\u(t)\|^{2}_{2} + \int_{0}^{t}\left(\mu\|\u(s)\|^2_{\Vbb}+2\alpha\|\u(s)\|^2_{2}+2\beta\|\u(s)\|^{\varpi+1}_{\varpi+1}\right)\drm s - \|\u_0\|_2^2 - \widehat{\Mrm}t -t-2  \right) \geq R \right\}
		\nonumber\\ 
		&  \leq \Pbb \left\{  \sup_{t \geq T} \left( \mathcal{M}(t)-t-2 \right) \right\}
		\nonumber\\ 
		& \leq c_q \sum_{m\geq \lfloor T\rfloor} \frac{\Ebb \left[\{[\mathcal{M}](t)(m+1)\}^{\frac{q}{2}}\right]}{(R+m+2)^q}
		\nonumber\\ 
		& \leq c_q \sum_{m\geq \lfloor T\rfloor} \frac{1}{(R+m+2)^q} \Ebb \left[\left\{C_{K,L} \int_0^{m+1} (1+ \|\u(s)\|^4_2 ) \drm s\right\}^{\frac{q}{2}}\right]
		\nonumber\\ 
		& \leq C_{q,K,L} \sum_{m\geq \lfloor T\rfloor} \frac{(m+1)^{\frac{q-2}{2}}}{(R+m+2)^q} \Ebb \left[ \int_0^{m+1} (1+ \|\u(s)\|^{2q}_2 ) \drm s\right]
		\nonumber\\ 
		& \leq C_{q,K,L} \left[1+ \|\u_0\|^{2q}_{2} + \frac{\Mrm}{q (\mu\lambda_1+2\alpha)}\right] \sum_{m\geq \lfloor T\rfloor} \frac{(m+1)^{\frac{q}{2}}}{(R+m+2)^q} 
		\nonumber\\ 
		& \leq C_{q,K,L} \left[1+ \|\u_0\|^{2q}_{2} + \frac{\Mrm}{q (\mu\lambda_1+2\alpha)}\right] \sum_{m\geq \lfloor T\rfloor} \frac{1}{(R+m+2)^{\frac{q}{2}}} 
		\nonumber\\ 
		& \leq C_{q,K,L} \left[1+ \|\u_0\|^{2q}_{2} + \frac{\Mrm}{q (\mu\lambda_1+2\alpha)}\right]  \frac{1}{(R+T)^{\frac{q}{2}-1}}.
	\end{align*}
This completes the proof of \textbf{Case 2.}
\end{proof} 

\begin{lemma}\label{lem-prob-esti-2}
	Suppose that $\varpi = 3$.	Under  Hypothesis \ref{hyp-noise}, the solution $\u$ of \eqref{SCBFE} satisfies the following estimate:
		\begin{align}\label{eqn-Prob-Est-varpi=3}
			& \Pbb\left\{\sup_{t \geq T} \left( \|\u(t)\|^{2}_{2} + \frac{\mu}{2} \int_{0}^{t}\|\u(s)\|^2_{\Vbb} \drm s - \|\u_0\|^{2}_{2}  - \widehat{\Mrm}t  \right) \geq R \right\}
			\leq 
			\begin{cases}
				e^{-  \frac{1}{8K} \left(\frac{\mu\lambda_1}{2} + 2\alpha\right) R}, & \text{ for } L=0; \\
				e^{- \min \left\{ \frac{\beta}{4L |\Ocal|^{\frac12}}, \frac{1}{8K} \left(\frac{\mu\lambda_1}{2} + 2\alpha\right)\right\} R}, & \text{ for } L>0;
			\end{cases}
		\end{align}
		for all $T\geq 0$ and $R>0$, where $\widehat{\Mrm}$ is given by \eqref{eqn-widehatM}.
\end{lemma}
\begin{proof}
	From \eqref{PrEs1} and \eqref{PrEs2}, we have 
	\begin{align*}
		& \|\u(t)\|^{2}_{2} + \int_{0}^{t}\left(\mu\|\u(s)\|^2_{\Vbb}+2\alpha\|\u(s)\|^2_{2}+2\beta\|\u(s)\|^{4}_{4}\right)\drm s
		  \leq \|\u_0\|^{2}_{2}  + \widehat{\Mrm}t  
		+   \mathcal{M}(t)
	\end{align*}
with
\begin{align*}
	[\mathcal{M}](t) & \leq  8 \int_{0}^{t}  \left[L\|\u(s)\|_2^4  + K\|\u(s)\|_2^2 \right] \drm s  \leq  8L |\Ocal|^{\frac12} \int_{0}^{t} \|\u(s)\|_4^4 \drm s + 8K \int_0^t \|\u(s)\|_2^2  \drm s .
\end{align*}
This helps us to get 
\begin{align*}
	 \|\u(t)\|^{2}_{2} +  \frac{\mu}{2} \int_{0}^{t}\|\u(s)\|^2_{\Vbb} \drm s -\|\u_0\|^{2}_{2} - \widehat{\Mrm} t
	& \leq    \mathcal{M}(t) -   \int_{0}^{t}\left(\frac{\mu}{2}\|\u(s)\|^2_{\Vbb}+2\alpha\|\u(s)\|^2_{2}+2\beta\|\u(s)\|^{4}_{4}\right)\drm s
	\nonumber\\ 
	& \leq    \mathcal{M}(t) -  \left(\frac{\mu\lambda_1}{2} + 2\alpha\right)  \int_{0}^{t}\|\u(s)\|^2_{2}\drm s  - 2\beta \int_{0}^{t}\|\u(s)\|^{4}_{4} \drm s
	\nonumber\\ 
	& \leq    \begin{cases}
		\mathcal{M}(t) - \frac{1}{8K} \left(\frac{\mu\lambda_1}{2} + 2\alpha\right)[\mathcal{M}](t), & \text{ for } L=0;\\
		\mathcal{M}(t) - \min \left\{ \frac{2\beta}{8L |\Ocal|^{\frac12}}, \frac{1}{8K} \left(\frac{\mu\lambda_1}{2} + 2\alpha\right)\right\} [\mathcal{M}](t), & \text{ for } L>0,
	\end{cases}
\end{align*}
which implies by the exponential martingale inequality \cite[Proposition 3.1]{Glatt-Holtz_2014_Arxiv} that \eqref{eqn-Prob-Est-varpi=3} holds.
This completes the proof.
\end{proof}

\section{Continuous data assimilation: 2D Subcritical SCBFEs}\label{Sec5}

For $d=2$ and $\varpi = 1$, the result is expected to coincide with that of \cite{BFLZ_2025_Arxiv}; hence, we consider only the case $\varpi >1$.  

The convergence analysis is carried out in expectation, so that the convergence results hold in the mean-square sense. Under suitable conditions on the observational resolution $\h$ and the nudging parameter $\sigma$, we prove that
\[
\mathbb{E}\bigl[\|\u(t)-\Urm(t)\|_2^2\bigr] \to 0
\quad \text{as } t \to +\infty.
\]
For $1 < \varpi < 3$, this convergence is exponentially fast in the case of additive noise, whereas in the multiplicative noise setting it is polynomial of arbitrary order $p \in (0,+\infty)$. In the context of data assimilation, the estimate $\mathbb{E}\bigl[\|\u(t)-\Urm(t)\|_2^2\bigr]$ is commonly referred to as a Foias-Prodi estimate in expectation. This estimate shows that the control (nudging) term $\sigma\,\Rrm_{\h}\bigl(\Urm(t)-\u(t)\bigr)$, when the parameters $\sigma$ and $\h$ are chosen appropriately, enforces synchronization in expectation between the solution $\u$ of \eqref{1} and the solution $\Urm$ of the corresponding data assimilation system as $t \to +\infty$.

We initially derive a result for a general stopping time, postponing the selection of a suitable one to a later stage.

\begin{lemma}\label{lem-Difference-1}
	Assume that $\Grm$ satisfies Hypothesis \ref{hyp-noise} and $\Rrm_{\h}$ satisfies \eqref{eqn-data-inter}. Suppose that $\u$ and $\Urm$ are the solutions to systems \eqref{SCBFE} and \eqref{SCBFE-CDA}, respectively. If
	\begin{equation*}
		0<\sigma \leq \frac{  \mu }{c_0 \h^2 },
	\end{equation*}
	where $c_0$ is the constant appearing in \eqref{eqn-data-inter}, then, for any $\delta\geq 0$, any stopping time $\tau$ and any $t \geq 0$, the following estimate is satisfied:
	\begin{align}\label{eqn-Difference-1}
		\mathbb{E}\left[e^{\left(\frac{2\alpha + \sigma}{1+\delta} - L \right)(t\wedge \tau) - \frac{1}{\mu} \int_0^{t \wedge \tau}\|\u(s)\|^2_{\Vbb}\, \drm s}\left\| \u(t\wedge \tau)-\Urm(t\wedge \tau) 
		\right\|_{2}^2\right] 
		&\leq \left\|\u_0 - \Urm_0 \right\|_{2}^2.
	\end{align}     
\end{lemma}

\begin{proof}
Let us define	$\w(\cdot):=\Urm(\cdot)-\u(\cdot)$, then $\w(\cdot)$ satisfies
	\begin{equation*}
	\left\{
		\begin{aligned}
			\drm \w(t) + \big[\mu \Acal\w(t) + \alpha \w(t)  &    + \left( \B(\Urm(t),\w(t)) + \B(\w(t),\u(t)) \right)  + \beta\left( \K(\Urm(t)) - \K(\u(t)) \right)  \big]\, \drm t  \\
			& = - \sigma \Pcal_{\frac{\varpi+1}{\varpi}} \Rrm_{\h}(\w(t)) + [\Grm(\Urm(t))-\Grm(\u(t))]\, \drm\Wrm(t), \;\;\;\;\;\;\;\;\;\;\;\;\;\;\;\;\;\;\; t>0  \\ 
			\w(0)& =\Urm_0 - \u_0.
		\end{aligned}
	\right.
	\end{equation*}
	Let us apply the It\^o formula to the process $\|\w(\cdot)\|_2^2$, use $\eqref{b0}_1$ and obtain $\mathbb{P}$-a.s.
	\begin{align}\label{Diff-1}
		& \|\w(t)\|^2_2 + 2\mu \int_0^t \|\w(s)\|^2_{\Vbb}\,\drm s  + 2\alpha \int_0^t \|\w(s)\|^2_2 \,\drm s + 2 \beta \int_0^t \left\langle\w(s) , \K(\Urm(s)) - \K(\u(s)) \right\rangle \drm s
		\nonumber\\
		&  =  \|\w(0)\|^2_2 - 2 \int_0^t \big[\left\langle \w (s), \B(\w(s), \u(s)) \right\rangle   + \sigma \left( \w (s),  \Rrm_{\h}(\w(s)) \right)\big] \drm s  
		\nonumber\\ 
		& \quad + \int_0^t \|\Grm(s,\Urm(s))-\Grm(s,\u(s))\|^2_{\mathcal{L}_{\Qrm}} \drm s  +  2 \int_0^t \big( \w (s), [\Grm(s,\Urm(s))-\Grm(s,\u(s))]\, \drm\Wrm(s) \big)
		\nonumber\\
		&  = \|\w(0)\|^2_2 - 2\sigma  \int_0^t \|\w(s)\|^2_2 \,\drm s - 2 \int_0^t \big[ b( \w(s), \u(s), \w (s) )    + \sigma \left( \w (s),  \Rrm_{\h}(\w(s)) - \w(s) \right)\big] \drm s  
		\nonumber\\ 
		& \quad + \int_0^t \|\Grm(s,\Urm(s))-\Grm(s,\u(s))\|^2_{\mathcal{L}_{\Qrm}} \drm s  +  2 \int_0^t \big( \w (s), [\Grm(s, \Urm(s))-\Grm(s, \u(s))]\, \drm\Wrm(s) \big).
	\end{align}
	We write from \eqref{2.23} and Hypothesis \ref{hyp-noise}, respectively, that
	\begin{align}\label{Diff-2}
		& \beta \||\Urm|(\Urm-\u)\|_{2}^2  +  \beta \||\u|(\Urm-\u)\|_{2}^2 \leq  2\beta \langle\Urm-\u, \K(\Urm)-\K(\u)\rangle ,
	\end{align}
	and
	\begin{align}\label{Diff-3}
		\|\Grm(s,\Urm)-\Grm(s,\u)\|^2_{\mathcal{L}_{\Qrm}} \leq L \|\w\|_2^2.
	\end{align}
Using \eqref{eqn-data-inter}, and H\"older's, Ladyzhenskaya's  and Young's inequalities, we get 
	\begin{align}\label{Diff-4}
		2\left|   b( \w, \u, \w )   + \sigma \left( \w,  \Rrm_{\h}(\w) - \w \right)  \right| 
		& \leq  2\left|   b( \w, \u, \w )  \right|  + 2 \left| \sigma \left( \w, \Rrm_{\h}(\w) - \w \right)  \right| 
		\nonumber\\ 
		&  \leq   2 \|\u\|_{\Vbb} \|\w\|_4^2  + 2  \sigma \| \w\|_2 \| \Rrm_{\h}(\w) - \w\|_2 
		\nonumber\\ 
		&  \leq    2^{\frac{3}{2}} \|\u\|_{\Vbb} \|\w\|_2 \|\w\|_{\Vbb}  +   \sigma \| \w\|^2_2 + \sigma \| \Rrm_{\h}(\w) - \w\|^2_2 
		\nonumber\\ 
		&  \leq  \mu \|\w\|^2_{\Vbb} +  \frac{2}{\mu} \|\u\|_{\Vbb}^2 \|\w\|_2^2 +   \sigma \| \w\|^2_2 + \sigma c_0 \h^2 \|\w\|^2_{\Vbb} .
	\end{align}

	A combination of \eqref{Diff-1}-\eqref{Diff-4} gives us the following:
	\begin{align}\label{Diff-5}
		& \|\w(t)\|^2_2 + (\mu-\sigma c_0 \h^2) \int_0^t \|\w(s)\|^2_{\Vbb}\,\drm s  +  \int_0^t  \left(2\alpha+\sigma - \frac{2}{\mu} \|\u(s)\|_{\Vbb}^2  - L\right) \|\w(s)\|^2_2 \,\drm s 
		\nonumber\\
		&  \leq  \|\w(0)\|^2_2 + 2 \int_0^t \big( \w (s), [\Grm(\Urm(s))-\Grm(\u(s))]\, \drm\Wrm(s) \big).	
	\end{align}
	
	For $0<\sigma \le \frac{  \mu }{c_0 \h^2 }$,	let us apply the  It\^o formula to the process $e^{\varrho(\cdot)}\|\w(\cdot)\|_{2}^2$, where 
	\begin{align}\label{eqn-varrho}
		\varrho(t)= \left(\frac{2\alpha+\sigma}{1+\delta} -  L \right)t  -  \frac{2}{\mu}\int_0^t\|\u(s)\|_{\Vbb}^2\d s\ \text{ so that }\ \varrho'(t)= \frac{2\alpha+\sigma}{1+\delta} -  L - \frac{2}{\mu}\|\u(t)\|_{\Vbb}^2, \ \text{ for a.e. } t,
	\end{align} 
	to obtain, $\Pbb$-a.s.,
	\begin{align}\label{Diff-6}
		& e^{\varrho(t)} \|\w(t)\|^2_2  + \frac{(2\alpha+\sigma)\delta}{1+\delta} \int_{0}^{t} e^{\varrho(s)} \|\w(s)\|^2_2 \drm s
	 \leq \|\w(0)\|^2_2 + 2 \int_0^t e^{\varrho(s)}  \big( \w (s), [\Grm(\Urm(s))-\Grm(\u(s))]\, \drm\Wrm(s) \big).		
	\end{align}
Taking expectations on both sides of \eqref{Diff-6} and noting that the stochastic integral is a (local) martingale, yield
	\begin{align*}
		& \Ebb \left[e^{\varrho(t\land \tau)} \|\w(t\land \tau)\|^2_2 \right] 
		\leq  \|\w(0)\|^2_2,
	\end{align*}
for any $\delta\geq 0$, any stopping time $\tau$ and any $t \geq 0$. This completes the proof.	
\end{proof}

To control the integrating factor in \eqref{eqn-Difference-1}, we introduce a suitable stopping time. For $R$, $\uppi$, $\delta$, $\sigma>0,$ define
\begin{align}\label{eqn-stopping-time}
	\uptau_{R,\uppi,\delta,\sigma}:= \inf \left\{ s\geq 0: \frac{2}{\mu}\int_0^s\|\u(t)\|^2_{\Vbb} \drm t + \left(L - \frac{2\alpha+\sigma}{(1+\delta)^2}\right)s - \uppi \geq R \right\},
\end{align}
and $\uptau_{R,\uppi,\delta,\sigma}=+\infty$ is the set is empty, that is, if 
\begin{align*}
	\frac{2}{\mu}\int_0^s\|\u(t)\|^2_{\Vbb} \drm t + \left(L - \frac{2\alpha+\sigma}{(1+\delta)^2}\right)s - \uppi < R, \;\;\; \text{ for all } s\geq 0.
\end{align*}
The parameter $\uppi$ is introduced to keep track of the dependence on the initial data $\u_0$, the external forcing $\f$, and other related constants throughout the subsequent analysis. 

Combining the definition of the stopping time $\uptau_{R,\uppi,\delta,\sigma}$ with Lemma \ref{lem-Difference-1}, we arrive immediately at the following result.
\begin{corollary}\label{cor-stop-time}
	Under the Hypothesis of Lemma \ref{lem-Difference-1}, for any $\u_0$, $\Urm_0\in\Hbb$ and any $R$, $\uppi$, $\delta>0$, we get
	\begin{align*}
		\Ebb\left[ \mathbf{1}_{(\uptau_{R,\uppi,\delta,\sigma}=+\infty)}  \|\u(t)-\Urm(t)\|^2_{2} \right] \leq e^{R+\uppi - \frac{\delta(2\alpha+\sigma)}{(1+\delta)^2}t} \|\u_0-\Urm_0\|^2_2,
	\end{align*}
where $\mathbf{1}_{\cdot}$ represent the characteristic function.
\end{corollary}
\begin{proof}
	The argument is analogous to the proof of \cite[Corollary 4.4]{BFLZ_2025_Arxiv}.
\end{proof}

In the result below, we derive, for appropriately chosen parameters $\uppi$, $\delta$, and $\sigma$, an estimate of $\Pbb(\uptau_{R,\uppi,\delta,\sigma}<+\infty)$ depending on the parameter $R$.

\begin{proposition}\label{prop-prob-esti-1}
	Let $R>0$ and $\delta>0$ be arbitrary given, and Hypothesis \ref{hyp-noise} is satisfied. For the solution of system \eqref{SCBFE} with initial data $\u_0$, we consider the stopping time $\uptau_{R,\uppi,\delta,\sigma}$ defined in \eqref{eqn-stopping-time}. Then, 
	\begin{itemize}
		\item [(i)] For $L=0$, if 
		\begin{align}\label{pi-1}
			\uppi  \geq  \frac{4}{\mu^2} \|\u_0\|^2_2
		\end{align} 
		and 
		\begin{align}\label{sigma-1}
		2\alpha+\sigma \geq (1+\delta)^2 	\frac{4}{\mu^2}	\left[K  + \frac{\|\f\|_{\Vbb^*}^2 }{\mu}\right]    
		\end{align} 
		then
		\begin{align}\label{Prob-stopping-time-1}
			\Pbb(\uptau_{R,\uppi,\delta,\sigma}<+\infty) \leq  e^{- \frac{\mu^2}{32K} \left(\frac{\mu\lambda_1}{2} + 2\alpha\right) R}
		\end{align} 
		\item [(ii)]  For $L>0$, if 
		\begin{align}\label{pi-2}
				\uppi  \geq  \frac{2}{\mu^2} \left[\|\u_0\|^2_2 + 2\right]
		\end{align} 
		and 
		\begin{align}\label{sigma-2}
		2\alpha+\sigma \geq (1+\delta)^2 	\left\{\frac{2}{\mu^2}	\left[\widehat{\Mrm}+1\right]  + L\right\} 
		\end{align} 
		then
		\begin{align}\label{Prob-stopping-time-2}
			\Pbb(\uptau_{R,\uppi,\delta,\sigma}<+\infty) \leq C_{p,K,L}  \left(\frac{2}{\mu^2}\right)^{p} \left[1+ \|\u_0\|^{4(p+1)}_{2} + \frac{\Mrm}{q (\mu\lambda_1+2\alpha)}\right]  \frac{1}{R^p},
		\end{align} 
	for any $p\in (0,+\infty)$, where $\Mrm, \widehat{\Mrm}, C_{p,K,L}>0$ are the constants appearing in \eqref{eqn-moments-2p-u}, \eqref{eqn-widehatM} and \eqref{eqn-Prob-Est-varpi2}, respectively.
	\end{itemize}
\end{proposition}
\begin{proof}
	By invoking Lemma \ref{lem-prob-esti-1} and using the assumptions on the parameters $\uppi$ and $\sigma$ given in \eqref{pi-1}–\eqref{sigma-1} and \eqref{pi-2}–\eqref{sigma-2}, the proof is immediate. We refer to \cite[Proposition 4.3]{BZ_DCDS} and \cite[Proposition 4.5]{BFLZ_2025_Arxiv} for analogous arguments.
\end{proof}

Now, we state and prove the main result of this section.

\begin{theorem}\label{MT-Subcritical}
	Assume that Hypothesis \ref{hyp-noise} is fulfilled and $\Rrm_{\h}$ satisfies \eqref{eqn-data-inter}. Suppose that $\u$  and $\Urm$ are the solutions to the systems \eqref{SCBFE} and \eqref{SCBFE-CDA}, respectively. Then
	\begin{enumerate}
		\item For $L=0$,  if  $\h$  and $\sigma$ satisfy the condition 
		\begin{align}\label{condition-on-sigma-1}
			\frac{4}{\mu^2}	\left[K  + \frac{\|\f\|_{\Vbb^*}^2 }{\mu}\right]   <  2\alpha+\sigma \leq  2 \alpha + \frac{  \mu }{c_0 \h^2 },
		\end{align}
		then
		$\mathbb{E}\left[\|\u(t)-\Urm(t)\|_2^2\right] \to 0$ exponentially fast as $t \to + \infty$.
		\\
		\item  For $L>0$, if $\h$  and $\sigma$ satisfy the condition 
		\begin{equation}\label{condition-on-sigma-2}
			\frac{2}{\mu^2}	\left[K  + \frac{\|\f\|_{\Vbb^*}^2 }{\mu} + \left\{\frac{\beta(\varpi+1)}{2}\right\}^{-\frac{2}{\varpi-1}} \left\{\frac{\varpi+1}{\varpi-1}\right\}^{-1} \left\{L \right\}^{\frac{\varpi+1}{\varpi-1}}  |\Ocal|+1\right]  + L <  2\alpha+\sigma  \leq  2 \alpha + \frac{  \mu }{c_0 \h^2 }, 
		\end{equation}
		then $\mathbb{E}\left[\|\u(t)-\Urm(t)\|_2^2\right] \to 0$ $p$-polynomially fast as $t \to + \infty$, for any power $p \in (0, +\infty)$.
	\end{enumerate}
\end{theorem}
\begin{proof}
	The proof is same in both the cases but requires different estimates from the Proposition \ref{prop-prob-esti-1}. Making use of H\"older's inequality, Corollary \ref{cor-stop-time} and Lemmas \ref{lem-pth-moments}-\ref{lem-pth-moments-Urm}, we find, for any $R, \uppi, \delta>0$ and $0<\sigma \leq \frac{\mu}{c_0 \h^2 }$, 
	\begin{align}\label{MT-1-1}
		\mathbb{E} \left[ \|\u(t)-\Urm(t)\|^2_2\right]
		& = \mathbb{E} \left[\mathbf{1}_{(\uptau_{R, \uppi, \delta, \sigma =+\infty  })} \|\u(t)-\Urm(t)\|^2_2\right] + \mathbb{E} \left[\mathbf{1}_{(\uptau_{R, \uppi, \delta, \sigma <+\infty})} \|\u(t)-\Urm(t)\|^2_2\right]
		\notag \\
		& \leq e^{ R + \uppi - \frac{\delta (2\alpha+\sigma)}{(1+\delta)^2} t} \|\u_0-\Urm_0\|^2_2 + \left( \mathbb{P}(\uptau_{R, \uppi, \delta, \sigma}<+\infty)\right)^{\frac12} \left(\mathbb{E}\left[\|\u(t)-\Urm(t)\|^4_2\right] \right)^{\frac12}
		\notag\\
		& \leq
		C\left(1+\|\u_0\|^2_2 +  \|\Urm_0\|^2_2 \right)\left(\left( \mathbb{P}(\uptau_{R, \uppi, \delta, \sigma}<+\infty)\right)^{\frac12}+ e^{ R + \uppi - \frac{\delta(2\alpha+\sigma)}{(1+\delta)^2} t}\right),
	\end{align}
	where $C>0$ is a constant independent of $t$.
	
	We proceed by fixing suitable parameters $R$, $\uppi$, $\delta$, and $\sigma$, and then exploiting the previously established bounds for $\uptau_{R,\uppi,\delta,\sigma}$. From the Proposition \ref{prop-prob-esti-1}, we gain
	\begin{align}\label{MT-1-2}
		(\mathbb{P}(\uptau_{R, \uppi, \delta, \sigma}<+\infty))^{\frac12} 
		\leq 
		\begin{cases}
		e^{-C_1 R}, & \text{ for } L=0 \text{ with } \eqref{condition-on-sigma-1},\\
		C_2 R^{-\frac{p}{2}}, & \text{ for } L>0 \text{ with } \eqref{condition-on-sigma-2},
		\end{cases} 
	\end{align}
where the constants $C_1$ and $C_2$ are independent of $R$ and $t$. Now combining \eqref{MT-1-1} and \eqref{MT-1-2}, and choosing $R= \frac{\delta(2\alpha+\sigma)}{2(1+\delta)^2}t$, for each $t > 0$, we immediately complete the proof. 
\end{proof}

\begin{remark}\label{Rem-Subcritical}
	For $\varpi \in[1,3)$, assume that, instead of the linear growth condition, the noise coefficient $\Grm$ satisfies the following sublinear growth assumption:
	\begin{align*}
		\|\Grm(t,\u)\|_{\mathcal{L}_{\Qrm}}^2
		\leq K + \widetilde{K}\,\|\u\|_2^{\gamma-1},
		\qquad \text{for } \gamma \in [1,\varpi].
	\end{align*}
	Under this assumption, one can also obtain exponential convergence of
	\[
	\mathbb{E}\bigl[\|\u(t)-\Urm(t)\|_2^2\bigr] \to 0
	\quad \text{as } t \to +\infty,
	\]
	in the case $L>0$. We direct the reader to the next section, where the case $\varpi = 3$ is analyzed in detail.
\end{remark}

\section{Continuous data assimilation: 2D and 3D Critical SCBFEs}\label{Sec6}

In this section, we analyze exclusively the case $\varpi = 3$ in two and three dimensions. This case merits separate consideration, as it leads to exponential convergence of $\mathbb{E}\left[\|\u(t)-\Urm(t)\|_2^2\right] \to 0$ as $t \to +\infty$, irrespective of whether $L=0$ or $L>0$.

\begin{lemma}
		Under  Hypothesis \ref{hyp-noise} on $\Grm$ and assumption \eqref{eqn-data-inter} $\Rrm_{\h}$, suppose that $\u$ and $\Urm$ are the unique solutions of  systems \eqref{SCBFE} and \eqref{SCBFE-CDA}, respectively. If 
	\begin{itemize}
		\item  [(i)] for $d\in\{2,3\}$ and $2\beta\mu>1$, $\sigma$ and $\h$ are such that 
		\begin{align}\label{condition-on-sigma-3}
			 0 <  \sigma \leq  \frac{  1}{c_0 \h^2 }\left(2\mu-\frac{1}{\beta}\right),
		\end{align}
		where $c_0$ is the constant appearing in \eqref{eqn-data-inter}, then for any $t\geq0$
		\begin{align}\label{eqn-Difference-2}
			& \Ebb \left[ \|\u(t)-\Urm(t)\|^2_2 \right] 
			\leq  \|\u_0-\Urm_0\|^2_2  e^{-(2\alpha+\sigma  - L)t} .	
		\end{align}
		\item  [(ii)] for $d=2$, if $\sigma$ and $\h$ are such that 
		\begin{align}\label{condition-on-sigma-4}
			0  <  \sigma \leq   \frac{  \mu }{c_0 \h^2 }
		\end{align}
		where $c_0$ is the constant appearing in \eqref{eqn-data-inter}, then for any $\delta\ge 0$, any stopping time $\tau$ and any $t \ge 0$:
		\begin{align}\label{eqn-Difference-3}
			\mathbb{E}\left[e^{\left(\frac{2\alpha + \sigma}{1+\delta} - L \right)(t\wedge \tau) - \frac{1}{\mu} \int_0^{t \wedge \tau}\|\u(s)\|^2_{\Vbb}\, \drm s}\left\| \u(t\wedge \tau)-\Urm(t\wedge \tau) 
			\right\|_{2}^2\right] 
			&\leq \left\|\u_0 - \Urm_0 \right\|_{2}^2.
		\end{align}
	\end{itemize}
\end{lemma}

\begin{proof}
	We consider both the cases one by one.	
	\vskip2mm
	\noindent
	\textbf{\textit{Case} (i).} \textbf{For $d\in\{2,3\}$ and $2\beta\mu>1$ with \eqref{condition-on-sigma-3}}. 	Using $\eqref{b0}_2$ and \eqref{eqn-data-inter}, and H\"older's, Ladyzhenskaya's  and Young's inequalities, we get
	\begin{align}\label{Diff-7}
		 2\left|   b( \w, \u, \w )   + \sigma \left( \w, \Rrm_{\h}(\w) - \w \right)  \right| 
		& \leq  2\left|   b( \w, \w, \u )  \right|  + 2 \left| \sigma \left( \w,  \Rrm_{\h}(\w) - \w \right)  \right| 
		\nonumber\\ 
		&  \leq  \frac{1}{\beta} \|\w\|^2_{\Vbb} + \beta  \||\u|\w\|_2^2  + 2  \sigma \| \w\|_2 \| \Rrm_{\h}(\w) - \w\|_2 
		\nonumber\\ 
		&  \leq  \frac{1}{\beta} \|\w\|^2_{\Vbb} + \beta  \||\u|\w\|_2^2    +   \sigma \| \w\|^2_2 + \sigma \| \Rrm_{\h}(\w) - \w\|^2_2 
		\nonumber\\ 
		&  \leq  \frac{1}{\beta} \|\w\|^2_{\Vbb} + \beta  \||\u|\w\|_2^2    +   \sigma \| \w\|^2_2 + \sigma c_0 \h^2 \|\w\|^2_{\Vbb} 
	\end{align}	
	By combining \eqref{Diff-1}-\eqref{Diff-3} and \eqref{Diff-7}, we reach at 
		\begin{align}\label{Diff-8}
		& \|\w(t)\|^2_2 + (2\mu- \frac{1}{\beta}-\sigma c_0 \h^2) \int_0^t \|\w(s)\|^2_{\Vbb}\,\drm s  + (2\alpha+\sigma  - L) \int_0^t \|\w(s)\|^2_2 \,\drm s 
		\nonumber\\
		&  \leq  \|\w(0)\|^2_2 + 2 \int_0^t \big( \w (s), [\Grm(\Urm(s))-\Grm(\u(s))]\, \drm\Wrm(s) \big).	
	\end{align}
	
	For $0< \sigma \leq 2\alpha + \frac{  1 }{c_0 \h^2 }\left(2\mu- \frac{1}{\beta}\right)$, taking expectations on both sides of \eqref{Diff-8} and noting that the stochastic integral is a (local) martingale, yield
		\begin{align}
		& \Ebb\left[ \|\w(t)\|^2_2\right]  + (2\alpha+\sigma  - L) \int_0^t \Ebb \left[\|\w(s)\|^2_2\right] \,\drm s 
		  \leq  \|\w(0)\|^2_2 .	
	\end{align}
which immediately provides \eqref{eqn-Difference-2}. 

The proof of assertion (ii) is obtained by repeating the argument of the proof of Lemma \ref{lem-Difference-1}, and for this reason we omit the details. Hence, we complete the proof.
\end{proof}

\begin{proposition}
For $d=2$ and $\varpi=3$, let $R>0$ and $\delta>0$ be arbitrary given, and Hypothesis \ref{hyp-noise} is satisfied. For the solution of system \eqref{SCBFE} with initial data $\u_0$, we consider the stopping time $\uptau_{R,\uppi,\delta,\sigma}$ defined in \eqref{eqn-stopping-time}. If
		\begin{align}\label{pi-3}
			\uppi  \geq  \frac{4}{\mu^2} \|\u_0\|^2_2
		\end{align} 
		and 
		\begin{align}\label{sigma-3}
		2\alpha+\sigma \geq \left\{\frac{4\widehat{\Mrm}}{\mu^2}	 + L\right\} (1+\delta)^2, 
		\end{align} 
	where $\widehat{\Mrm}$ is give in \eqref{eqn-widehatM},	then
		\begin{align*}
			\Pbb(\uptau_{R,\uppi,\delta,\sigma}<+\infty)
			\leq 
			\begin{cases}
				e^{-  \frac{1}{8K} \left(\frac{\mu\lambda_1}{2} + 2\alpha\right) \frac{\mu^2R}{4}}, & \text{ for } L=0; \\
				e^{- \min \left\{ \frac{\beta}{4L |\Ocal|^{\frac12}}, \frac{1}{8K} \left(\frac{\mu\lambda_1}{2} + 2\alpha\right)\right\} \frac{\mu^2R}{4}}, & \text{ for } L>0.
			\end{cases}
		\end{align*} 
\end{proposition}
\begin{proof}
	By invoking Lemma \ref{lem-prob-esti-2} and using the assumptions on the parameters $\uppi$ and $\sigma$ given in \eqref{pi-3}–\eqref{sigma-3}, the proof is immediate.
\end{proof}

%
%
%
%
%
%
%
%
%
%
%
%
%
%
%
%
%
%

Now, we state and prove the main result of this section.

\begin{theorem}\label{MT-Critical}
	Assume that Hypothesis \ref{hyp-noise} is fulfilled and $\Rrm_{\h}$ satisfies \eqref{eqn-data-inter}. Suppose that $\u$  and $\Urm$ are the solutions to the systems \eqref{SCBFE} and \eqref{SCBFE-CDA}, respectively. Then
	\begin{enumerate}
		\item for $d\in\{2,3\}$ and $2\beta\mu>1$, $\sigma$ and $\h$ are such that 
		\begin{align}\label{condition-on-sigma-5}
			L  <  2\alpha+\sigma \leq  2\alpha +  \frac{  1}{c_0 \h^2 }\left(2\mu-\frac{1}{\beta}\right),
		\end{align}
		then $\mathbb{E}\left[\|\u(t)-\Urm(t)\|_2^2\right] \to 0$ exponentially fast as $t \to + \infty$.
		\\
		\item  for $d=2$, if $\sigma$ and $\h$ are such that 
		\begin{equation}\label{condition-on-sigma-6}
			 \frac{4}{\mu^2} \left\{K  + \frac{\|\f\|_{\Vbb^*}^2 }{\mu} + \frac{L^2|\Ocal|}{4\beta}  \right\}	 + L  <  2\alpha+\sigma  \leq  2 \alpha + \frac{  \mu }{c_0 \h^2 }, 
		\end{equation}
		then $\mathbb{E}\left[\|\u(t)-\Urm(t)\|_2^2\right] \to 0$ exponentially fast as $t \to + \infty$.
	\end{enumerate}
\end{theorem}
\begin{proof}
See the proof of Theorem \ref{MT-Subcritical}.
\end{proof}

\section{Continuous data assimilation: 2D and 3D Supercritical SCBFEs}\label{Sec7}

In this section, we consider $\varpi>3$ and show that $\mathbb{E}\left[\|\u(t)-\Urm(t)\|_2^2\right] \to 0$ exponentially fast as $t \to +\infty$, irrespective of whether $L=0$ or $L>0$.

\begin{theorem}\label{MT-Supercritical}
	Under  Hypothesis \ref{hyp-noise} and assumption \ref{eqn-data-inter} on $\Grm$ and $\Rrm_{\h}$, let $\u$ and $\Urm$ are the unique solutions of systems \eqref{SCBFE} and \eqref{SCBFE-CDA}, respectively. If 
	\begin{itemize}
		\item  [(i)] $\sigma$ and $\h$ are such that 
		\begin{align}\label{condition-on-sigma-7}
			2 \widehat{\boldsymbol{\upvarpi}} + L < 2\alpha + \sigma \leq 2\alpha + \frac{  \mu}{c_0 \h^2 }
		\end{align}
	where $\widehat{\boldsymbol{\upvarpi}}$ and $c_0$ are the constants appearing in \eqref{eqn-upvarpi} and \eqref{eqn-data-inter}, respectively, then $\mathbb{E}\left[\|\u(t)-\Urm(t)\|_2^2\right] \to 0$ exponentially fast as $t \to + \infty$.
\\
\item  [(ii)] for $2\beta\mu>1$, $\sigma$ and $\h$ are such that 
\begin{align}\label{condition-on-sigma-8}
	\beta + L < 2\alpha + \sigma \leq 2\alpha + \frac{  1}{c_0 \h^2 }\left(2\mu-\frac{1}{\beta}\right)
\end{align}
where $c_0$ is the constant appearing in,  then $\mathbb{E}\left[\|\u(t)-\Urm(t)\|_2^2\right] \to 0$ exponentially fast as $t \to + \infty$.

	\end{itemize}
\end{theorem}
\begin{proof}
	Let us discuss both cases separately.	
\vskip2mm
\noindent
\textbf{\textit{Case} (i).} \textbf{With \eqref{condition-on-sigma-7}}.   From \eqref{bilinear-etimate-1}, \eqref{eqn-data-inter}, and H\"older's and Young's inequalities, we achieve
	\begin{align}\label{Diff-9}
		 2\left|   b( \w, \u, \w )   + \sigma \left( \w,  \Rrm_{\h}(\w) - \w \right)  \right| 
		& \leq  2\left|   b( \w, \w, \u )  \right|  + 2 \left| \sigma \left( \w,  \Rrm_{\h}(\w) - \w \right)  \right| 
		\nonumber\\ 
		&  \leq  \mu \|\w\|^2_{\Vbb} + \frac{\beta}{2}  \||\u|^{\frac{\varpi -1}{2}}\w\|_2^2  + 2 \widehat{\boldsymbol{\upvarpi}}  \|\w\|_2^2 + 2  \sigma \| \w\|_2 \|\Rrm_{\h}(\w) - \w\|_2 
		\nonumber\\ 
		&  \leq  \mu \|\w\|^2_{\Vbb} + \frac{\beta}{2}  \||\u|^{\frac{\varpi -1}{2}}\w\|_2^2  + 2 \widehat{\boldsymbol{\upvarpi}}  \|\w\|_2^2 +   \sigma \| \w\|^2_2 + \sigma \| \Rrm_{\h}(\w) - \w\|^2_2 
		\nonumber\\ 
		&  \leq  \mu \|\w\|^2_{\Vbb} + \frac{\beta}{2}  \||\u|^{\frac{\varpi -1}{2}}\w\|_2^2  + 2 \widehat{\boldsymbol{\upvarpi}}  \|\w\|_2^2 +   \sigma \| \w\|^2_2 + \sigma c_0 \h^2 \|\w\|^2_{\Vbb} 
	\end{align}
	where $\widehat{\boldsymbol{\upvarpi}}$ is defined in \eqref{eqn-upvarpi}.

	Combining \eqref{Diff-1}-\eqref{Diff-3} and \eqref{Diff-9}, we find
	\begin{align}\label{Diff-10}
		& \|\w(t)\|^2_2 + (\mu-\sigma c_0 \h^2) \int_0^t \|\w(s)\|^2_{\Vbb}\,\drm s  + (2\alpha+\sigma - 2 \widehat{\boldsymbol{\upvarpi}} - L) \int_0^t \|\w(s)\|^2_2 \,\drm s 
		\nonumber\\
		&  \leq  \|\w(0)\|^2_2 + 2 \int_0^t \big( \w (s), [\Grm(\Urm(s))-\Grm(\u(s))]\, \drm\Wrm(s) \big).	
	\end{align}

For $2 \widehat{\boldsymbol{\upvarpi}} + L < 2\alpha + \sigma \leq 2\alpha + \frac{  \mu}{c_0 \h^2 }$, taking expectations on both sides of \eqref{Diff-10}, and noting that the stochastic integral is a (local) martingale, yield
\begin{align}
	& \Ebb\left[ \|\w(t)\|^2_2\right]  + (2\alpha+\sigma -2 \widehat{\boldsymbol{\upvarpi}}  - L) \int_0^t \Ebb \left[\|\w(s)\|^2_2\right] \,\drm s 
	\leq  \|\w(0)\|^2_2,
\end{align}
which, by Gronwall's lemma, immediately provides 
	\begin{align}
		& \Ebb \left[ \|\w(t)\|^2_2 \right] 
		\leq  \|\w(0)\|^2_2  e^{-(2\alpha+\sigma - 2 \widehat{\boldsymbol{\upvarpi}} - L)t} \to 0 \text{ as } t \to +\infty.
\end{align}	
	\vskip2mm
	\noindent
	\textbf{\textit{Case} (ii).} \textbf{For $2\beta\mu>1$ with \eqref{condition-on-sigma-8}}. From \eqref{bilinear-etimate-2}, \eqref{eqn-data-inter}, and H\"older's and Young's inequalities, we achieve
	\begin{align}\label{Diff-11}
		& 2\left|   b( \w, \u, \w )   + \sigma \left( \w,  \Rrm_{\h}(\w) - \w \right)  \right| 
		\leq   \frac{1}{\beta} \|\w\|_{\Vbb}^2 + \beta \||\u|^{\frac{\varpi-1}{2}}\w\|_{2}^2 + \beta \|\w\|_{2}^2 +   \sigma \| \w\|^2_2 + \sigma c_0 \h^2 \|\w\|^2_{\Vbb}.
	\end{align}

		Combining \eqref{Diff-1}-\eqref{Diff-3} and \eqref{Diff-11}, we find
\begin{align}\label{Diff-12}
	& \|\w(t)\|^2_2 + (2\mu - \frac{1}{\beta} - \sigma c_0 \h^2) \int_0^t \|\w(s)\|^2_{\Vbb}\,\drm s  + (2\alpha+\sigma - \beta - L) \int_0^t \|\w(s)\|^2_2 \,\drm s  
	\nonumber\\
	&  \leq  \|\w(0)\|^2_2 + 2 \int_0^t \big( \w (s), [\Grm(\Urm(s))-\Grm(\u(s))]\, \drm\Wrm(s) \big).	
\end{align}

For $\beta + L < 2\alpha + \sigma \leq 2\alpha + \frac{  1}{c_0 \h^2 }\left(2\mu - \frac{1}{\beta}\right)$, taking expectations on both sides of \eqref{Diff-12}, and noting that the stochastic integral is a (local) martingale, yield
\begin{align*}
	& \Ebb\left[ \|\w(t)\|^2_2\right]  + (2\alpha+\sigma -\beta  - L) \int_0^t \Ebb \left[\|\w(s)\|^2_2\right] \,\drm s 
	\leq  \|\w(0)\|^2_2,
\end{align*}
which, by Gronwall's lemma, immediately provides 
\begin{align*}
	& \Ebb \left[ \|\w(t)\|^2_2 \right] 
	\leq  \|\w(0)\|^2_2  e^{-(2\alpha+\sigma - \beta - L)t} \to 0 \text{ as } t \to +\infty.
\end{align*}	
This concludes the proof.
\end{proof}

\section{Pathwise convergence}\label{Sec8}

In Theorems \ref{MT-Subcritical}, \ref{MT-Critical}, and \ref{MT-Supercritical}, we establish that
\[
\lim_{t \to +\infty} \Ebb \left[ \|\u(t)-\Urm(t)\|_2^2 \right] = 0.
\]
As a consequence, there exists a subsequence $\{t_{n}\}_{n\in\N}$ such that $t_n \to +\infty$ as $n \to +\infty$, along which $\|\u(t)-\Urm(t)\|_2^2$ converges $\Pbb$-a.s. Moreover, as shown in \cite[Corollary 4.6]{BFLZ_2025_Arxiv}, for any sequence $\{t_n\}_{n\in\mathbb{N}}$ with $t_n \to +\infty$ as $n \to +\infty$, one can obtain that
\[
 \|\u(t_n)-\Urm(t_n) \|_2^2 \to 0
\quad \Pbb\text{-a.s.}
\]
We emphasize, however, that this result alone is not sufficient to conclude pathwise convergence, that is,
\[
 \|\u(t)-\Urm(t) \|_2^2 \to 0
\quad \text{as } t \to +\infty,\ \Pbb\text{-a.s.}
\]

\begin{corollary}
\label{pathwise_mult}
Under the same assumption as in Theorems \ref{MT-Subcritical}, \ref{MT-Critical}, and \ref{MT-Supercritical}, for any sequence of times $\{t_{n}\}_{n\in\N}$ such that $t_n \to +\infty$ as $n \to +\infty$, we get 
 \[
\mathbb{P}\left( \lim_{ n \to + \infty}\|\u(t_n)-\Urm(t_n)\|^2_2 =0\right) = 1.
 \]
\end{corollary}
\begin{proof}
See the proof of \cite[Corollary 4.6]{BFLZ_2025_Arxiv}.
\end{proof}

We now turn to the pathwise analysis, in which convergence results are understood in the $\Pbb$-a.s. sense. Throughout this part, we assume that $\Grm$ is independent of $\u$. Our main result is stated as follows.

\subsection{2D SCBFEs with $\varpi\in[1,3]$}

Let us first provide an auxiliary lemma which will be used in sequel.

\begin{lemma}
\label{data-estimate-1}
Let $\Grm\in \mathcal{L}_{\Qrm}(\Hbb)$ and let $\Rrm_{\h}$ satisfy \eqref{eqn-data-inter}. 
Let $\u$ and $\Urm$ be the solutions to the systems \eqref{SCBFE} and \eqref{SCBFE-CDA}, respectively. If $\sigma$ and $\h$ satisfy
\begin{equation}
\label{condi-1-mu}
0<\sigma \le \frac{  \mu }{c_0 \h^2 } 
\end{equation}
 where $c_0$ is the constant that appears in estimate \eqref{eqn-data-inter}, then $\Pbb$-a.s.  we have
\begin{align}\label{error-u-U}
\left\| \u(t)-\Urm(t) \right\|_{2}^2 &\leq \left\|\u_0 - \Urm_0 \right\|_{2}^2 \cdot  \exp \left(-(2\alpha+\sigma) t + \frac{2}{\mu} \displaystyle \int_{0}^t  \left\| \u(s) \right\|_{\Vbb}^2 \drm s   \right), \qquad \text{ for any }t>0.
\end{align} 
\end{lemma}
\begin{proof}
	Since $\Grm$ is independent of $\u$ and consequently $L=0$, we infer from \eqref{Diff-5} that $\w:= \u-\Urm$ satisfies
	\begin{align}\label{Diff-13}
		& \|\w(t)\|^2_2 + (\mu-\sigma c_0 \h^2) \int_0^t \|\w(s)\|^2_{\Vbb}\,\drm s  +  \int_0^t  \left(2\alpha+\sigma - \frac{2}{\mu} \|\u(s)\|_{\Vbb}^2  \right) \|\w(s)\|^2_2 \,\drm s 
		  \leq  \|\w(0)\|^2_2.
	\end{align}
Hence, an application of Gronwall's inequality completes the proof.
\end{proof}

Now, we establish the main result of this subsection.

\begin{theorem}\label{pathwise_data_ass}
	Let $\Grm\in \mathcal{L}_{\Qrm}(\Hbb)$ and let $\Rrm_{\h}$ satisfy \eqref{eqn-data-inter}. Let $\u$ and $\Urm$ be the solutions to the systems \eqref{SCBFE} and \eqref{SCBFE-CDA}, respectively. If $\sigma$ and $\h$ satisfy
	\begin{equation}\label{cond_mu}
		\frac{4}{\mu^2} \left[\|\Grm\|^2_{\mathcal{L}_{\Qrm}} + \frac{\|\f\|_{\Vbb^*}^2 }{\mu} \right] < 2\alpha+ \sigma \leq 2\alpha+ \frac{  \mu }{c_0 \h^2 }, 
	\end{equation}
	then
	$\|\u(t)-\Urm(t)\|_2 \to 0$ exponentially fast as $t \to + \infty$, $\mathbb{P}$-a.s.
\end{theorem}

\begin{proof}
In view of \eqref{eqn-Prob-Est-varpi1} and \eqref{eqn-Prob-Est-varpi=3}, we have for $\varpi\in[1,3]$
\begin{align}
	& \Pbb\left\{\sup_{t \geq T} \left( \|\u(t)\|^{2}_{2} + \int_{0}^{t}\left(\frac{\mu}{2}\|\u(s)\|^2_{\Vbb}+2\beta\|\u(s)\|^{\varpi+1}_{\varpi+1}\right)\drm s - \|\u_0\|^{2}_{2}  - \left[\|\Grm\|^2_{\mathcal{L}_{\Qrm}}  + \frac{\|\f\|_{\Vbb^*}^2 }{\mu} \right]t  \right) \geq R \right\}
	\nonumber\\ 
	& \leq e^{- \frac{1}{8K} \left(\frac{\mu\lambda_1}{2} + 2\alpha\right) R};
\end{align}
for all $T\geq 0$ and $R>0$, which implies
\begin{align}
	& \Pbb\left\{ \limsup_{t \to +\infty} \frac{2}{\mu t} \int_{0}^{t}\|\u(s)\|^2_{\Vbb} \drm s  \leq  \frac{4}{\mu^2} \left[\|\Grm\|^2_{\mathcal{L}_{\Qrm}}  + \frac{\|\f\|_{\Vbb^*}^2 }{\mu} \right]  \right\} = 1.
\end{align}
Therefore, for any initial velocity $\u_0\in \Hbb$, 
\begin{align}\label{limit-t}
	\limsup_{t \to +\infty} \frac{2}{\mu t} \int_{0}^{t}\|\u(s)\|^2_{\Vbb} \drm s  \leq  \frac{4}{\mu^2} \left[\|\Grm\|^2_{\mathcal{L}_{\Qrm}} + \frac{\|\f\|_{\Vbb^*}^2 }{\mu} \right], \;\;\; \Pbb\text{-a.s.}
\end{align}
From  \eqref{cond_mu} there exists a constant $\varepsilon>0$ such that
\[
2\alpha + \sigma-  \frac{4}{\mu^2} \left[\|\Grm\|^2_{\mathcal{L}_{\Qrm}} + \frac{\|\f\|_{\Vbb^*}^2 }{\mu} \right] - \varepsilon>0.
\]
This gives along with \eqref{limit-t}
\begin{align}
	& \lim_{t \to +\infty}\exp \left(-(2\alpha+\sigma) t + \frac{2}{\mu} \displaystyle \int_{0}^t  \left\| \u(s) \right\|_{\Vbb}^2 \drm s   \right)
	\nonumber\\
	&  \leq
	\lim_{t \to +\infty}\exp \left( - \varepsilon t - \frac{4}{\mu^2} \left[\|\Grm\|^2_{\mathcal{L}_{\Qrm}} + \frac{\|\f\|_{\Vbb^*}^2 }{\mu} \right] t + \frac{2}{\mu} \displaystyle \int_{0}^t  \left\| \u(s) \right\|_{\Vbb}^2 \drm s   \right)  =0,
\end{align}
and by \eqref{error-u-U} we also have that $\u-\Urm$ vanishes for large times. 
Particularly,  we see that there exists a time $t_0>0$ such that for any $t>t_0$, we have
\[
\|\u(t)-\Urm(t)\|_2^2 \leq 
\|\u_0-\Urm_0\|^2_2
  \exp\left(-\left(2\alpha+ \sigma- \frac{4}{\mu^2} \left[\|\Grm\|^2_{\mathcal{L}_{\Qrm}} + \frac{\|\f\|_{\Vbb^*}^2 }{\mu} \right]-\varepsilon \right)t\right)
\]
showing the exponential convergence rate. This completes the proof.
\end{proof}

\subsection{2D and 3D SCBFEs for $\varpi\geq3$}

\begin{theorem}\label{pathwise_data_ass-geq3}
	Let $\Grm\in \mathcal{L}_{\Qrm}(\Hbb)$ and let $\Rrm_{\h}$ satisfy \eqref{eqn-data-inter}. Let $\u$ and $\Urm$ be the solutions to the systems \eqref{SCBFE} and \eqref{SCBFE-CDA}, respectively. If 
	\begin{enumerate}
		\item for $d\in\{2,3\}$ with $\varpi=3$ and $2\beta\mu>1$, $\sigma$ and $\h$ satisfy
		\begin{align}\label{cond_mu-1}
			0 <  \sigma \leq  \frac{  \mu }{c_0 \h^2 }\left(2\mu - \frac1\beta\right);
		\end{align}
	\item for $d\in\{2,3\}$ with $\varpi>3$, $\sigma$ and $\h$ satisfy
	\begin{align}\label{cond_mu-2}
		2 \widehat{\boldsymbol{\upvarpi}}  < 2\alpha + \sigma \leq 2\alpha + \frac{  \mu}{c_0 \h^2 };
	\end{align}
where $\widehat{\boldsymbol{\upvarpi}}$ is given by \eqref{eqn-upvarpi},
\item for $d\in\{2,3\}$ with $\varpi>3$ and $2\beta\mu>1$, $\sigma$ and $\h$ satisfy
\begin{align}\label{cond_mu-3}
	\beta <  2\alpha + \sigma \leq  2\alpha + \frac{  \mu }{c_0 \h^2 }\left(2\mu - \frac1\beta\right);
\end{align}
	\end{enumerate}
	then
	$\|\u(t)-\Urm(t)\|_2 \to 0$ exponentially fast as $t \to + \infty$, $\mathbb{P}$-a.s.
\end{theorem}

\begin{proof}
Since $\Grm$ is independent of $\u$ and consequently $L=0$,	in view of \eqref{Diff-8}, \eqref{Diff-10} and \eqref{Diff-12}, respectively, we write
	\begin{align*}
		 \|\w(t)\|^2_2  - \|\w(0)\|^2_2  \leq  -  
		 \begin{cases}
		 	(2\alpha+\sigma ) \int_0^t \|\w(s)\|^2_2 \,\drm s, & \text{ for } d\in\{2,3\} \text{ with } \varpi=3 \text{ and } \eqref{cond_mu-1};\\
		 	(2\alpha+\sigma - 2 \widehat{\boldsymbol{\upvarpi}} )  \int_0^t \|\w(s)\|^2_2 \,\drm s, & \text{ for } d\in\{2,3\} \text{ with } \varpi>3 \text{ and } \eqref{cond_mu-2};\\
		 	(2\alpha+\sigma - \beta )  \int_0^t \|\w(s)\|^2_2 \,\drm s, & \text{ for } d\in\{2,3\} \text{ with } \varpi>3 \text{ and } \eqref{cond_mu-3}.
		 \end{cases}	
	\end{align*}
Hence, an application of Gronwall's inequality conclude the proof.
\end{proof}

\medskip\noindent
{\bf Acknowledgments:}     This work is funded by national funds through the FCT - Fundação para a Ciência e a Tecnologia, I.P., under the scope of the projects UID/297/2025 and UID/PRR/297/2025 (Center for Mathematics and Applications - NOVA Math).  K. Kinra wish to thank Dr. M.T. Mohan for suggesting this problem.

\medskip\noindent
\textbf{Data availability:} No data was used for the research described in the article.

\medskip\noindent
\textbf{Declarations}: During the preparation of this work, the authors have not used AI tools.

\medskip\noindent
\textbf{Author Contributions}: The sole author wrote, and edited the entire manuscript.

\medskip\noindent
\textbf{Conflict of interest:} The author declares no conflict of interest.

\appendix


\end{document}